\documentclass[11pt,a4paper,reqno]{amsart}
\usepackage{amsthm,amsmath,amsfonts,amssymb,amsxtra,appendix,bookmark,dsfont,latexsym,bm,delarray,euscript,amsgen,amsbsy,amsopn,amscd,latexsym,mathrsfs}
\usepackage[usenames,dvipsnames]{xcolor}
\definecolor{darkblue}{RGB}{0,0,170}
\definecolor{brickred}{RGB}{200,0,0}
\definecolor{darkblu}{RGB}{100,40,150}
\definecolor{xanhle}{RGB}{0,0,0}
\usepackage[colorinlistoftodos]{todonotes}
\hypersetup{colorlinks=true, linkcolor=darkblue, citecolor=brickred}
\usepackage{enumitem}
\usepackage{todonotes}
\usepackage[margin=1in]{geometry}

\makeatletter
\def\@seccntformat#1{%
	\protect\textup{%
		\protect\@secnumfont
		\expandafter\protect\csname format#1\endcsname 
		\csname the#1\endcsname
		\protect\@secnumpunct
	}%
}

\makeatother

\makeatletter
\def\th@plain{%
	\thm@notefont{}
	\slshape 
}
\def\th@definition{%
	\thm@notefont{}
	\normalfont 
}
\makeatother
\theoremstyle{plain}
\newtheorem{theorem}{Theorem}
\numberwithin{theorem}{section}
\newtheorem{lemma}[theorem]{Lemma}
\newtheorem{proposition}[theorem]{Proposition}

\newtheorem{corollary}[theorem]{Corollary}

\theoremstyle{definition}
\newtheorem{definition}[theorem]{Definition}
\newtheorem{remark}[theorem]{Remark}

\newcommand{\vertiii}[1]{{\left\vert\kern-0.3ex\left\vert\kern-0.3ex\left\vert #1 
    \right\vert\kern-0.3ex\right\vert\kern-0.3ex\right\vert}}
\newcommand{\Hbb}{\mathbb{H}}
\newcommand{\Bb}{\mathbb{B}}

\newcommand{\R}{\mathbb{R}}
\newcommand{\N}{\mathbb{N}}

\newcommand{\T}{\mathbb{T}}
\newcommand{\M}{\mathcal{M}}
\renewcommand{\L}{\mathbb{L}}
\newcommand{\Gbb}{\mathbb{G}}

\newcommand{\dist}{\hbox{dist}}

\newcommand{\va}[1]{\left\{\begin{array}{r@{\text{ }}ll}#1\end{array}\right.}
\newcommand{\inner}[1]{\left \langle #1 \right\rangle}

\newcommand{\norm}[1]{\left\| #1\right\|}
\DeclareMathOperator{\loc}{loc}
\DeclareMathOperator{\sgn}{sign}

\DeclareMathOperator{\SFL}{SFL}
\DeclareMathOperator{\CFL}{CFL}
\DeclareMathOperator{\RFL}{RFL}

\DeclareMathOperator{\supp}{supp}

\usepackage{amsthm}

\numberwithin{equation}{section}
\usepackage[normalem]{ulem}
\newcommand{\revision}[2]{\sout{#1}{\color{darkblue}#2}}

\hypersetup{colorlinks=true, linkcolor=darkblue, citecolor=brickred}

\newcommand{\Ibb}{\mathbb{I}}
\newcommand{\D}{\mathcal{D}}

\newcommand{\B}{\mathcal{E}}

\newcommand{\CS}{\mathcal{S}}
\newcommand{\CB}{\mathcal{B}}
\newcommand{\CF}{\mathcal{F}}
\newcommand{\CO}{\mathcal{O}}

\newcommand\1{{\ensuremath {\mathds 1} }}

\begin{document}

\title[Semilinear nonlocal elliptic equations]{Compactness of Green operators  with applications \\ to semilinear nonlocal elliptic equations}
\author[P.-T. Huynh]{Phuoc-Truong Huynh}
\address{P.-T. Huynh,  Department of Mathematics, Alpen-Adria-Universit\"{a}t Klagenfurt,  Klagenfurt, Austria.}
\email{phuoc.huynh@aau.at}

\author[P.-T. Nguyen]{Phuoc-Tai Nguyen}
\address{P.-T. Nguyen, Department of Mathematics and Statistics, Masaryk University, Brno, Czech Republic.}
\email{ptnguyen@math.muni.cz}


\begin{abstract} 
In this paper, we consider a class of integro-differential operators $\mathbb{L}$ posed on a $C^2$ bounded domain $\Omega \subset \mathbb{R}^N$ with appropriate homogeneous Dirichlet conditions where each of which admits an inverse operator commonly known as the Green operator $\Gbb^{\Omega}$.  Under mild conditions on $\mathbb{L}$ and its Green operator, we establish various sharp compactness of $\mathbb{G}^{\Omega}$ involving weighted Lebesgue spaces and weighted measure spaces. These results are then employed to prove the solvability for semilinear elliptic equation $\mathbb{L} u + g(u) = \mu$ in $\Omega$ with boundary condition $u=0$ on $\partial \Omega$ or exterior condition $u=0$ in $\mathbb{R}^N \setminus \Omega$ if applicable, where $\mu$ is a Radon measure on $\Omega$ and $g: \R \to \R$ is a nondecreasing continuous function satisfying a subcriticality integral condition. When $g(t)=|t|^{p-1}t$ with $p>1$, we provide a sharp sufficient condition expressed in terms of suitable Bessel capacities for the existence of a solution.  The contribution of the paper consists of (i) developing novel unified techniques which allow to treat various types of fractional operators and (ii) obtaining sharp compactness and existence results in weighted spaces, which refine and extend several related results in the literature.
\bigskip
	
\noindent\textit{Keywords: integro-differential operators, compactness, Green function, Kato-type inequalities, critical exponents, weak-dual solutions, measure data, Bessel capacities, Riesz potentials.}
	
\bigskip
	
\noindent\textit{2020 Mathematics Subject Classification: 35J61, 35B33, 35B65, 35R06, 35R11, 35D30, 35J08.}

\end{abstract}

\maketitle

\tableofcontents

%
%

\section{Introduction}
Over the last few decades, there have been remarkable advancements in the research of integral-differential operators, which are prevalent in diverse fields such as probability, physics, finance, and ecology. The nonlocal intrinsic nature of these operators also attracts significant attention from the PDE community due to the occurrence of new phenomena associated with them. In particular, their inverse operators,  known to exist when we consider suitable boundary or external conditions,  play an important role in the theory of elliptic equations; therefore the approach involving the inverse operators has been studied in several works (see e.g.    \cite{CheVer_2014,  Aba_2015} for equations involving the standard fractional Laplacian and \cite{AbaDup_2017} for equations involving the spectral fractional Laplacian). 

Recently, unified techniques have been developed based on a fine analysis to derive various properties of such operators and their corresponding Green function, which is known as the kernel of the inverse operator when a suitable boundary/external condition is posed. Such techniques have been widely employed, initially in the investigation of fractional porous medium equations \cite{BonSirVaz_2015} and subsequently in the study of sharp boundary behavior to semilinear elliptic equations by Bonforte, Figalli, and V\'{a}zquez \cite{BonFigVaz_2018} and other works.

In the same spirit,  the present paper is devoted to the study a class of integral-differential linear operators $\L$ posed on a $C^2$ bounded domain $\Omega \subset \R^N$ ($N \geq 3$) together their inverse operators commonly referred to as the Green operators, denoted by $\Gbb^{\Omega}$, and represented by a Green functions $G^{\Omega}$. The detailed list of assumptions is given in Subsection \ref{sec:pre_assumption} and is illustrated by well-known fractional Laplacians, as discussed in Subsection \ref{sec:examples}.  The goal of the paper is twofold: to develop unified techniques to establish compactness properties of the Green operator $\Gbb^{\Omega}$ between various optimal weighted spaces under the given assumptions, and subsequently to apply these results in the study of the boundary value problem for semilinear elliptic equations with absorption term of the form
\begin{equation}\label{eq:equation_introduction}
\L u + g(u) = \mu \text{ in }\Omega,
\end{equation}
with homogeneous Dirichlet condition
\begin{equation} \label{eq:bdry-cond}
u = 0 \text{ on }\partial \Omega \quad \text{or}\quad \text{ in }\R^N \backslash \Omega \text{ if applicable}, 
\end{equation}
where $\mu$ is a weighted Radon measure on $\Omega$ and $g: \R \to \R$ is a continuous function.  

\subsection*{Overview of the literature}

The compactness of Green operators for both differential and integro-differential linear operators is known to be a useful tool in studying various types of associated nonlinear equations.  Typically in the literature this result is established by using two key factors: the regularity theory governing solutions to linear equations and compact embeddings between the involved function spaces.  For instance, in the case of the (minus) Laplacian  $-\Delta$, the associated Green operator $\Gbb^\Omega$ is compact from $\M(\Omega,\delta)$ into $L^q(\Omega,\delta)$ for any $q \in [1,\frac{N+1}{N})$, where $\delta(x)=\dist(x,\partial \Omega)$, due to the fact that $\Gbb: \M(\Omega,\delta) \to W^{1,q}(\Omega,\delta)$ is continuous (see \cite[Theorem 3.5]{Ver-handbook}) and the Sobolev compact embedding $W^{1,q}(\Omega,\delta)  \hookrightarrow  \hookrightarrow  L^q(\Omega,\delta)$ (the definitions of function spaces $L^q(\Omega,\delta^s)$  and measure spaces $\M(\Omega,\delta^s)$ are given in Subsection \ref{sec:result}).  A similar result was also established for the fractional Laplacian $(-\Delta)^s$, $s \in (0,1)$, defined by
$$ (-\Delta)^s u(x):=c_{N,s} \, \mathrm{P.V.} \int_{\R^N} \frac{u(x)-u(y)}{|x-y|^{N+2s}}dy,\quad x \in \Omega,$$
with $c_{N,s} > 0$ being a normalized constant and the abbreviation P.V. standing for ``in the principal value sense", see \eqref{Lepsilon}. More precisely, the compactness of  $\Gbb^{\Omega}: L^1(\Omega,\delta^s) \to L^q(\Omega)$ for any $q \in [1,\frac{N}{N-s})$ was proved by Chen and V\'eron \cite[Proposition 2.6]{CheVer_2014} by applying the Stampacchia's duality method and a fractional Sobolev embedding.  In the aforementioned results, compact embeddings from Sobolev spaces into Lebesgue spaces play a crucial role.  The compactness result for the Green operator of $(-\Delta)^s$, $s \in (0,1]$, was employed to study the boundary value problem 
\begin{equation} \label{prob:frac} \left\{  \begin{aligned}
		(-\Delta)^s u + g(u) &= \mu &&\text{ in }\Omega,\\ 
		u &= 0 &&\text{ on } \partial \Omega \text{ if } s =1 \text{ or in } \R^N \backslash \Omega \text{ if } s \in (0,1),
\end{aligned} \right. \end{equation}
where $g:\R \to \R$ is a nondecreasing, continuous \textit{absorption} with $g(0) = 0$. An abundant theory for problem \eqref{prob:frac} has been developed in function settings and then extended to measure frameworks (see e.g. \cite{BreStr_1973,Bre_1980,BenBre_2003,MarVer_2014,Pon_2016} for the local case $s=1$ and \cite{CheVer_2014,AbaDup_2017} for the nonlocal case $s \in (0,1)$). 
Among other results, it is noteworthy that by using the compactness results of $\Gbb^{\Omega}$, Chen and V\'eron \cite{CheVer_2014} proved the existence of a weak solution to \eqref{prob:frac} for all measures $\mu \in \M(\Omega,\delta^s)$ provided $f$ satisfies a sharp subcritical integrability condition (which includes the subcritical case $g(u)=|u|^{p-1}u$, $1<p<\frac{N+s}{N-s}$), and for a class of measures absolutely continuous with respect to a suitable capacity if $g(u)=|u|^{p-1}u$, $p \geq \frac{N+s}{N-s}$.


A great deal of effort has been devoted to the study of elliptic equations involving a large class of nonlocal operators characterized by estimates on the associated Green function. In this framework, Bonforte and V\'azquez  \cite{BonVaz_2015} introduced a notion of weak-dual solution which involves only the associated Green function and does not require to specify the meaning of the operators acting on test functions.  This research direction has been greatly pushed forward by numerous publications, including Bonforte, Figalli and V\'azquez \cite{BonFigVaz_2018},  Abatangelo, G\'omez-Castro and V\'azquez \cite{Aba_2019}, and  Chan, G\'omez-Castro and V\'azquez  \cite{GomVaz_2019}.  Lately, a work originated in attempt to study semilinear equations with source terms was carried out in \cite{TruongTai_2020} where existence and multiplicity results were established under a smallness condition on data. In parallel, significant contributions based on probabilistic approaches were given by Kim, Song and Vondra\v cek in \cite{KimSonVon_2020, KimSonVon_2020-1,KimSonVon_2019}. In addition, some semilinear integro-differential elliptic equations have also been treated in this fashion in several recent works \cite{BogJarKan_2011, BiVoWag_2021, Klimsiak_2023}.

\subsection*{Contributions} Motivated by the aforementioned works, in this paper we investigate the compactness of Green operators and semilinear nonlocal elliptic equations involving measures. Our main contributions are outlined below and will be stated in details in Section \ref{sec:result}.

(i) We obtain the compactness of the mapping $\mu \mapsto \Gbb^{\Omega}[\mu]$ from a weighted measure space into a weighted Lebesgue space. It is worth highlighting that the weighted measure space that we work on is the optimal class of data for which $\Gbb^{\Omega}[\mu]$ is finite a.e. in $\Omega$. Our approach is based on Riesz-Fr\'echet-Kolmogorov theorem and properties of $\Gbb^{\Omega}$ without using Sobolev embeddings, which justifies the difference our technique in comparison with those employed in \cite{Ver-handbook,CheVer_2014}.

(ii) We prove the existence results for problem \eqref{eq:equation_introduction}--\eqref{eq:bdry-cond} in the context of weak-dual solutions introduced above.  An important ingredient in the deriving the existence of solutions is the compactness of $\Gbb^{\Omega}$ mentioned in (i).  The uniqueness follows directly from variants of Kato's inequality in the context of weak-dual solutions which is proved by using the expression of $\mathbb{L}$ and fine properties of variational solutions to non-homogeneous linear equations. When $g$ is a purely power function, we establish a sharp sufficient condition expressed in terms of Bessel capacities for the existence of a solution. Our existence and uniqueness results cover those in \cite{Ver-handbook, CheVer_2014} and seems to be new in case of other fractional Laplacians.


\subsection*{Organization of the paper}
The rest of the paper is organized as follows. In Section \ref{sec:preliminaries}, we present the main assumptions on the class of operators and provide some examples. In Section \ref{sec:result}, we state the central results of this paper,  namely the compactness of the Green operator $\Gbb^{\Omega}$ associated to $\L$ and the solvability of problem \eqref{eq:equation_introduction}--\eqref{eq:bdry-cond}.  In Section \ref{sec:lineartheory}, we perform the proof of the compactness of $\Gbb^{\Omega}$.  Section \ref{sec:katoinequality} is devoted to the derivation of Kato type inequality. In Section \ref{sec:semilinear},  we prove the existence and uniqueness of solutions to problem \eqref{eq:equation_introduction}--\eqref{eq:bdry-cond}
by employing the results in Sections  \ref{sec:lineartheory} and \ref{sec:katoinequality}. In Section \ref{sec:sufficientcondition}, we focus on the power absorption and provide a sufficient condition in terms of appropriate Bessel capacities for the solvability of problem \eqref{eq:equation_introduction}--\eqref{eq:bdry-cond}. As a byproduct, in Appendix \ref{sec:boundarysolution}, we show the existence of a solution to equation \eqref{eq:equation_introduction} with an isolated boundary singularity.

\section{Main assumptions and examples}\label{sec:preliminaries}

Throughout this paper, we assume that $\Omega \subset \R^N$ ($N \ge 3$) is a $C^2$ bounded domain and let $\delta(x)$ be the distance from $x \in \Omega$ to $\R^N \backslash \Omega$. We denote by $c, C, c_1, c_2, \ldots$ positive constants that may vary from one appearance to another and depend only on the data. The notation $C = C(a,b,\ldots)$ indicates the dependence of constant $C$ on $a,b,\ldots$ For two functions $f$ and $g$, we write $f \lesssim g$ ($f \gtrsim g$) if there exists a constant $C > 0$ such that $f \le C g$ ($f \ge C g$). We write $f \asymp g$ if $f \lesssim g$ and $g \lesssim f$.  We denote $a\vee b := \max\{ a,b\}$ and $a\wedge b := \min \{ a, b\}$. For a function $f$, the positive part and the negative part of $f$ are $f^+= f \vee0 $ and $f^-=(-f) \vee0 $ respectively.

\subsection{Main assumptions}\label{sec:pre_assumption} In this subsection, we provide a list of assumptions on the class of admissible operators $\L$. Although some assumptions are similar to those in \cite{TruongTai_2020}, we use independent labels in this work to enhance clarity and ensure self-contained readability. This is also because we employ some different, yet stronger, assumptions in an effort to achieve potentially better results. We note that some similar conditions have also been introduced in \cite{BonFigVaz_2018,Aba_2019,ChaGomVaz_2019_2020}.

\noindent \textbf{Assumptions on $\L$.} We consider the class of nonlocal operators $\mathbb{L}$ with the following properties.

\begin{enumerate}[label= (L\arabic{enumi}),ref=L\arabic{enumi}]
\item \label{eq_L1} The operator $\L: C^{\infty}_c(\Omega) \to L^2(\Omega)$ is defined in terms of measurable functions $J$ and $B$ via the formula
\begin{equation} \label{LJB}
\L u(x) := \mathrm{P.V.} \int_{\Omega}J(x,y)[u(x)-u(y)]dy + B(x)u(x), \quad u \in C^{\infty}_c(\Omega), x\in \Omega.
\end{equation}
Here, the so-called jump kernel $J$ is nonnegative,  symmetric, finite on $(\Omega \times \Omega) \setminus D_{\Omega}$, where $D_{\Omega}:= \left\{ (x,x): x \in \Omega \right\}$, and satisfies
$$\int_{\Omega} \int_{\Omega} |x-y|^2 J(x,y) dx dy < +\infty,$$
and $B$ is nonnegative, locally bounded in $\Omega$.
\end{enumerate}
The abbreviation P.V. in \eqref{LJB} stands for ``in the principal value sense", namely
\begin{align*} 
\L u(x) = \lim_{\varepsilon \to 0^+} \L_{\varepsilon} u(x),
\end{align*}
where
\begin{equation} \label{Lepsilon}
\L_{\varepsilon} u (x):= \int_{\Omega} J(x,y)(u(x)-u(y))\chi_{\varepsilon}(|x-y|) dy + B(x)u(x)
\end{equation}
with 
$$
\chi_{\varepsilon}(t):= \va{ &0 \quad &\text{if } 0 \leq t \le \varepsilon, \\ &1 \quad &\text{if } t > \varepsilon.}
$$

\begin{enumerate}[label= (L\arabic{enumi}),ref=L\arabic{enumi}, resume]
\item \label{eq_L2bis} For every $u \in C^{\infty}_c(\Omega)$, there exists a function $\varphi_u \in L^1_{\loc}(\Omega)$ and $\varepsilon_0 > 0$ such that $|\L_{\varepsilon} u| \le \varphi_u$ a.e. in $\Omega$, for every $\varepsilon \in (0,\varepsilon_0]$.
\end{enumerate}

Under assumptions \eqref{eq_L1} and \eqref{eq_L2bis}, the following integration-by-parts formula
\begin{equation}\label{eq:integrationbypart1}
\begin{aligned}
\int_{\Omega} v \L u dx  &= \dfrac{1}{2} \int_{\Omega}\int_{\Omega} J(x,y)(u(x)-u(y))(v(x)-v(y)) dx dy + \int_{\Omega} B(x) u(x) v(x) dx \\
& = \int_{\Omega} u \L v dx
\end{aligned}
\end{equation}
holds for every $u,v \in C^{\infty}_c(\Omega)$, see Proposition \ref{prop:integrationbypart}.  Hence, $\L$ is symmetric and nonnegative on $C^{\infty}_c(\Omega)$. From this, we consider the bilinear form
\begin{equation} \label{bilinear}
{\color{darkblue}\B(u,v)}:= \dfrac{1}{2} \int_{\Omega}\int_{\Omega}  J(x,y)(u(x)-u(y))(v(x)-v(y))dx dy + \int_{\Omega} B(x)u(x)v(x) dx
\end{equation}
on the domain
\begin{align*} 
\D(\B):= \left\{ u \in L^2(\Omega): \B(u,u) < +\infty \right\}.
\end{align*}
It can be seen that $\D(\B)$ is nonempty since $C^{\infty}_c(\Omega) \subset \D(\B)$ by \eqref{eq:integrationbypart1} and assumption \eqref{eq_L1}. In particular, $\D(\B)$ is dense in $L^2(\Omega)$ with respect to the $L^2-$norm. Furthermore, by a standard argument, it can be checked that $\D(\B)$ is a complete space with respect to the norm induced by the inner product
\begin{equation}\label{eq:innerproduct}
\inner{ u,v}_{\D(\B)} := \inner{ u,v}_{L^2(\Omega)} + \B(u,v),\quad u,v \in \D(\B).
\end{equation}
On the other hand, since $\L$ is symmetric and nonnegative on $C^{\infty}_c(\Omega)$,  it admits the Friedrich extension $\widetilde{\L}$, which is nonnegative and self-adjoint.  For the reader's convenience, we recall the construction of the extension. Let $\Hbb(\Omega)$ be the closure of $C^{\infty}_c(\Omega)$ in $\D(\B)$ under the norm induced by \eqref{eq:innerproduct}.  Since $\D(\B)$ is complete, $\Hbb(\Omega)$ is also complete. By \cite[Theorem 4.4.2]{Dav_1995}, the form $\B|_{\Hbb(\Omega) \times \Hbb(\Omega)}$ is induced by a non-negative self-adjoint operator $\widetilde{\L}$ which satisfies
$$ {\B(u,v) = \inner{\L u, v}_{L^2(\Omega)},\quad \forall u,  v \in C^{\infty}_c(\Omega).}$$
By density,  the above equalities also hold for all $u \in C^{\infty}_c(\Omega)$ and $v \in \Hbb(\Omega)$.  Finally,  by \cite[Lemma 4.4.1]{Dav_1995}, $C^{\infty}_c(\Omega) \subset \D(\widetilde{\L})$ and $\widetilde{\L}u = \L u$ for all $u \in C^{\infty}_c(\Omega)$. This proves that $\widetilde{\L}$ is an extension of $\L$. In fact, $\Hbb(\Omega) = \D(\widetilde{\L}^{\frac{1}{2}})$,  where $\widetilde{\L}^{\frac{1}{2}}$ is the square root of $\widetilde{\L}$ which is also nonnegative and self-adjoint (see \cite[page 104]{Dav_1980}). \textit{Hereafter, without any confusion we denote by $\L$ the extension $\widetilde{\L}$.}

While the extension $\widetilde{\L}$ and the space $\Hbb(\Omega)$ are well-defined in a general framework, it is necessary to constrain them to a specific context in order to deploy certain properties. More precisely, we impose the following condition:
\begin{enumerate}[resume,label=(L\arabic{enumi}),ref=L\arabic{enumi}] 
	\item \label{eq_L2} There exists a constant $\Lambda  > 0$ such that
	\begin{align*}
	\norm{u}_{L^2(\Omega)}^2 \le \Lambda\inner{\L u, u}_{L^2(\Omega)},\quad \forall u \in C^{\infty}_c(\Omega). 
	\end{align*}
\end{enumerate}
\eqref{eq_L2} is the classical coercivity assumption.  Under this assumption and the fact that $C^{\infty}_c(\Omega)$ is dense in $\Hbb(\Omega)$, we have
\begin{equation}\label{eq:lowerbound2}
\norm{u}_{L^2(\Omega)}^2 \le \Lambda \B(u,u),\quad \forall u \in \Hbb(\Omega).
\end{equation}
By \eqref{eq:lowerbound2}, it can be seen that $\norm{u}_{\D(\B)}$ and $\norm{u}_{\Hbb(\Omega)}:= \sqrt{\B(u,u)}$ are equivalent norms on $\Hbb(\Omega)$. The norm $\| \cdot\|_{\Hbb(\Omega)}$ is associated to the inner product 
\begin{equation} \label{innerH} \inner{u,v}_{\Hbb(\Omega)} := \B(u,v), \quad u,v \in \Hbb(\Omega).
\end{equation}
In addition, we assume 
\begin{enumerate}[resume,label=(L\arabic{enumi}),ref=L\arabic{enumi}]
\item \label{eq_L4} $\Hbb(\Omega) = H^s_{00}(\Omega)$, where $H^s_{00}(\Omega)$ is the Banach space defined by 
\begin{equation} \label{H00}
	H^s_{00}(\Omega) := \left\{ u \in H^s(\Omega): \dfrac{u}{\delta^s} \in L^2(\Omega) \right\}
\end{equation}
with the norm
\begin{equation} \label{H00-norm}
	\norm{u}^2_{H^s_{00}(\Omega)} := \int_{\Omega} \left(1+\dfrac{1}{\delta^{2s}}\right)|u|^2 dx + \int_{\Omega} \int_{\Omega} \dfrac{|u(x)-u(y)|^2}{|x-y|^{N+2s}}dx dy.
\end{equation}
\end{enumerate}
In fact, $H^{s}_{00}(\Omega)$ is the space of functions in $H^s(\R^N)$ supported in $\Omega$, or equivalently, the trivial extension of functions in $H^s_{00}(\Omega)$ belongs to $H^s(\R^N)$ (see \cite[Lemma 1.3.2.6]{Gris_2011}). Furthermore, $C^{\infty}_c(\Omega)$ is a dense subset of $H^s_{00}(\Omega)$.  Alternatively, fractional Sobolev spaces can be viewed as interpolation spaces due to \cite[Chapter 1]{LioMag_1972}.  For more details on fractional Sobolev spaces, the reader is referred to \cite{Bha_2012,BonSirVaz_2015, NezPalVal_2012} and the list of references therein.

\noindent \textbf{Assumptions on the inverse of $\mathbb{L}$.} With the inner product on $\Hbb(\Omega)$ induced by $\mathcal{E}(\cdot,\cdot)$ in \eqref{innerH},  it can be seen that the map $\mathcal{E}: \Hbb(\Omega) \times \Hbb(\Omega) \to \R$ is continuous and coercive on $\Hbb(\Omega)$.  By the classical Lax-Milgram Theorem,  for every $f \in L^2(\Omega)$,  there exists a unique function $u \in \Hbb(\Omega)$ which is the solution of the problem $\L u = f$ in the sense that
$$ \mathcal{E}(u,v) = (f,  v),\quad \forall v \in \Hbb(\Omega) .$$
We denote the solution operator by $\Gbb^{\Omega}$,  which is also known as the Green operator of $\L$.  In other words, $\Gbb^{\Omega}$ is the inverse of $\L$, namely 
\begin{equation} \label{eq:homolinear}
	\L [\Gbb^{\Omega} [f]] = f \text{ in } L^2(\Omega).
\end{equation}

In the following,  we characterize the inverse of $\mathbb{L}$ by some assumptions outlined below.

\begin{enumerate}[label=(G\arabic{enumi}),ref=G\arabic{enumi}]
    \item \label{eq_G2} The operator $\Gbb^{\Omega}$ admits a Green function $G^{\Omega}$, namely
$$
\Gbb^{\Omega}[f](x) = \int_{\Omega} G^{\Omega}(x,y)f(y)dy,\quad x \in \Omega,
$$    
where $G^{\Omega}: (\Omega \times \Omega) \setminus D_{\Omega} \to (0,\infty)$ is Borel measurable, symmetric and satisfies
\begin{equation} \label{G-est} 
G^{\Omega} (x,y) \asymp \dfrac{1}{|x-y|^{N-2s}} \left( \dfrac{\delta(x)}{|x-y|} \wedge 1 \right)^{\gamma} \left( \dfrac{\delta(y)}{|x-y|} \wedge 1 \right)^{\gamma}, \quad x,y \in \Omega, x \neq y,
\end{equation}
with $s,\gamma \in (0,1]$ and $N>2s$.
\end{enumerate}

The existence of the Green function and assumption \eqref{eq_G2},  known to be satisfied for several local and nonlocal operators, are the same as \cite[assumptions (G1)--(G2)]{TruongTai_2020}, which enables us to establish quantitative estimates for the Green operator, as shown in \cite{BonFigVaz_2018, Aba_2019}. We note that  two-sided estimate \eqref{G-est} can be derived under a series of additional assumptions on $J$ and $B$ (see for instance the recent interesting results \cite[Theorem 1.1]{KimSongVond_2023} and \cite[Theorem 1.1]{KimSongVond_2024}), however, it is out of the scope  of  the present paper. In addition to \eqref{eq_G2},  to carry on further analysis in this work we require the following two assumptions on the Green function which concern its continuity as well as its behavior up to the boundary. 

\begin{enumerate}[label=(G\arabic{enumi}),ref=G\arabic{enumi}, resume]
\item \label{eq_G3bis} $G^{\Omega}$ is jointly continuous in $(\Omega \times \Omega) \backslash D_{\Omega}$.
\item \label{eq_G4} For every $x_0 \in \Omega$ and $z_0 \in \partial \Omega$, the limit
\begin{equation}\label{eq:martinkernel_def}
M^{\Omega}(x_0,z_0):= \lim_{(x,y) \to (x_0,z_0) } \dfrac{G^{\Omega}(x,y)}{\delta(y)^{\gamma}}
\end{equation}
exists and $M^{\Omega} : \Omega \times \partial \Omega \to \R^+$ is a continuous function.
\end{enumerate}

Assumptions similar to \eqref{eq_G3bis} and \eqref{eq_G4} have been considered in \cite[(K4) and (K5)]{Aba_2019}.  Assumptions \eqref{eq_G3bis} and \eqref{eq_G4} describe the behavior of $G^{\Omega}$ in $\Omega \times \Omega \setminus D_{\Omega}$ and in $\Omega \times \partial \Omega$ respectively, which in turn play an important role in the derivation of the compactness properties of $\Gbb^{\Omega}$ in different weighted function or measure spaces.  The function $M^{\Omega}$ defined in \eqref{eq:martinkernel_def} is closely related to the Martin kernel defined in \cite{SonVon_2003} and in \cite[formula (4.10)]{Aba_2019}. In addition, \eqref{G-est}  implies (see also e.g. \cite[Remark 4.10]{Aba_2019})
\begin{equation} \label{M-est} M^\Omega(x,z) \asymp \dfrac{ \delta(x)^\gamma}{|x-z|^{N-2s+2\gamma}}, \quad x \in \Omega, z \in \partial \Omega.
\end{equation}

\subsection{Some examples}\label{sec:examples}Typical nonlocal operators satisfying \eqref{eq_L1}--\eqref{eq_L4} and \eqref{eq_G2}--\eqref{eq_G4} to be kept in mind are the restricted fractional Laplacian, the spectral fractional Laplacian and the censored fractional Laplacian. We remark that,  for these examples, assumptions  \eqref{eq_L2}, \eqref{eq_L4} and \eqref{eq_G2} are satisfied, as pointed out in \cite{TruongTai_2020}, while assumption \eqref{eq_L2bis} is fulfilled in light of \cite[Proposition 3.4.]{KimSonVon_2020}.  Therefore, it is sufficient to verify only \eqref{eq_L1}, \eqref{eq_G3bis} and \eqref{eq_G4} in each of the following cases. 

 \medskip

\textbf{The restricted fractional Laplacian (RFL).} A well-known example of nonlocal operators satisfying the assumptions in Subsection \ref{sec:pre_assumption} is the restricted fractional Laplacian, or simply the fractional Laplacian, defined for $s \in (0,1)$ by
$$(-\Delta)^s_{\RFL} u(x): = c_{N,s}\, \mathrm{P.V.} \int_{\R^N} \frac{u(x)-u(y)}{|x-y|^{N+2s}}dy,\quad x \in \Omega, $$
restricted to functions that are zero outside $\Omega$ with $c_{N,s} > 0$ being a normalizing constant. This operator has been extensively studied in the literature, see for instance \cite{CheSon_1998,Aba_2015,RosSer_2014,SerVal_2012,SerVal_2014} and references therein. In this case, estimate \eqref{G-est} holds true with $\gamma = s$; see, e.g. \cite[Corollary 1.3]{CheSon_1998}, which implies that $G^{\Omega}(x_0, \cdot) \asymp \delta^{s}$ near $\partial \Omega$.

 Furthermore,  for $u \in C^{\infty}_c(\Omega)$,  one can write
\begin{align*}
(-\Delta)^s_{\RFL} u(x) 
=c_{N,s}\mathrm{P.V.}  \int_{\Omega} \frac{u(x)-u(y)}{|x-y|^{N+2s}}dy + c_{N,s}\int_{\R^N \backslash \Omega} \dfrac{u(x)}{|x-y|^{N+2s}}dy,
\end{align*}
hence \eqref{eq_L1} holds with
$$J(x,y) = \dfrac{c_{N,s}}{|x-y|^{N+2s}} \quad \text{and} \quad B(x) = c_{N,s}\int_{\R^N \backslash \Omega} \dfrac{1}{|x-y|^{N+2s}}dy,\quad x,y \in \Omega, \, x\neq y. $$
Next, one can see that \eqref{eq_G3bis} holds by \cite[page 467]{CheSon_1998}. Finally, \eqref{eq_G4} is fulfilled by \cite[Lemma 6.5]{CheSon_1998}.

\begin{remark} \label{justifyG4}
Alternatively, \eqref{eq_G4} can be verified by making use of the Martin kernel $\widetilde{M}^{\Omega}$ defined on $\Omega \times \partial \Omega$ (see \cite[page 278]{CheSon_1998_2}) by
\begin{equation} \label{Martin-classical}
\widetilde M^{\Omega}(x,z) := \lim_{\Omega \ni y \to z}\frac{G^\Omega(x,y)}{G^{\Omega}(x^*,y)},
\end{equation}	
where $x^* \in \Omega$ is a fixed reference point. Note that $\widetilde M^{\Omega}$ is continuous on $\Omega \times \partial \Omega$ (see \cite[Theorem 3.9]{CheSon_1998_2}).  On the other hand,  for every fixed $x_0 \in \Omega$, the limit
\begin{equation}\label{eq:limit_martin_abatangelo}
\lim_{y \to z \in \partial \Omega} \dfrac{G^{\Omega}(x_0,y)}{\delta(y)^{s}}
\end{equation}
exists and the limit function is continuous on $\Omega \times \partial \Omega$. Since \eqref{eq:limit_martin_abatangelo} is in fact \eqref{eq_G4} by fixing the first variable, we write $M^{\Omega}(x_0,z)$ for the limit function with an abuse of notation. Let $(x_0,z_0) \in \Omega \times \partial \Omega$. For any $(x,y) \in (\Omega \times \Omega) \setminus D_{\Omega}$, we write
\begin{equation}\label{eq:Martinboundary}
\dfrac{G^{\Omega}(x,y)}{\delta(y)^{s}} = \dfrac{G^{\Omega}(x,y)}{G^{\Omega}(x^*,y)}\cdot \dfrac{G^{\Omega}(x^*,y)}{\delta(y)^{s}}.
\end{equation}
Letting $(x,y) \to (x_0,z_0)$ and keeping in mind the definitions \eqref{eq:martinkernel_def}, \eqref{Martin-classical}, we obtain
$$ M^{\Omega}(x_0,z_0) = \widetilde M^{\Omega}(x_0,z_0)M^{\Omega}(x^*,z_0).
$$ 
Therefore $M^{\Omega}$ is well-defined and continuous on $\Omega \times \partial \Omega$.
\end{remark}

\medskip

\textbf{The spectral fractional Laplacian (SFL).} 
The spectral fractional Laplacian has been studied in, e.g., \cite{SonVon_2003,AbaDup_2017,DhiMaaZri_2011,CafSti_2016,BraColPabSan_2013}.
It is defined for $s \in (0,1)$ and $u \in C^{\infty}_c(\Omega)$ by
\[(-\Delta)^s_{\SFL} u(x):= \mathrm{P.V.} \int_{\Omega} [u(x)-u(y)]J_s(x,y)dy + B_s(x) u(x),\quad x \in \Omega,\]
where $J_s = J_s(x,y)$ and $B_s = B_s(x,y)$ are given by
$$J_s(x,y) := \dfrac{s}{\Gamma(1-s)} \int_0^{\infty} K_{\Omega} (t,x,y) \dfrac{dt}{t^{1+s}}dt, \quad x,y \in \Omega,$$
and 
$$B_s(x) := \dfrac{s}{\Gamma(1-s)} \int_0^{\infty} \left(1- \int_{\Omega} K_{\Omega} (t,x,y) dy\right) \dfrac{dt}{t^{1+s}},\quad x \in \Omega.$$
Hence, $K_{\Omega}$ is the heat kernel of the classical Dirichlet Laplacian $-\Delta$ in $\Omega$.  Thus, in this case, \eqref{eq_L1} is fulfilled with
$$J(x,y) = J_s(x,y) \quad \text{and} \quad B(x) = B_s(x),\quad x,y \in \Omega, \, x \neq y,$$
and estimate \eqref{G-est} holds true with $\gamma = 1$ (see, e.g., \cite[Remark 1.6]{AbaDup_2017}), which implies that $G^{\Omega}(x_0, \cdot) \asymp \delta$ near $\partial \Omega$. Finally,  we see that \eqref{eq_G3bis} follows by \cite[Proposition 2.2]{SonVon_2003}. Finally, assumption \eqref{eq_G4} can be verified by using a similar argument as in Remark \ref{justifyG4} by writing 
$$
\dfrac{G^{\Omega}(x,y)}{\delta(y)} = \dfrac{G^{\Omega}(x,y)}{G^{\Omega}(x^*,y)}\cdot \dfrac{G^{\Omega}(x^*,y)}{\delta(y)}
$$ 
as in \eqref{eq:Martinboundary}. Then we make use of the joint continuity of the (classical) Martin kernel of $(-\Delta)_{\SFL}^s$ on $\Omega \times \partial \Omega$ (see \cite[page 264]{DhiMaaZri_2011}) for the first term and the continuity of the Poison kernel of $(-\Delta)_{\SFL}^s$ with respect to the second argument on $\partial \Omega$ for the second term (see \cite[Lemma 14]{AbaDup_2017}). Note that since $\Omega$ is a $C^2$ domain, the Poisson kernel is involved in stead of the Martin kernel.

 \medskip

\textbf{The censored fractional Laplacian (CFL).} The censored fractional Laplacian is defined for $ s > \frac{1}{2}$ by
\begin{align*}
(-\Delta)^s_{\CFL} u(x) := c_{N,s}\, \mathrm{P.V.} \int_{\Omega} \dfrac{u(x)-u(y)}{|x-y|^{N+2s}}dy, \quad x \in \Omega.
\end{align*}
This operator has been studied in \cite{ChenKim_2002,BogBurChe_2003,Che_2018,Fal_2020}.  It can be seen that \eqref{eq_L1} holds for
\[J(x,y) = \dfrac{c_{N,s}}{|x-y|^{N+2s}} \text{ and }B(x) = 0, \quad x,y \in \Omega, \, x \neq y.\]
In this case, estimate \eqref{G-est} with $\gamma = 2s - 1$, which implies that $G^{\Omega}(x_0, \cdot) \asymp \delta^{2s-1}$ near $\partial \Omega$.  \eqref{eq_G3bis} holds true due to \cite[Theorem 1.1]{ChenKim_2002}). Finally, assumption \eqref{eq_G4} can be verified by using a similar argument as in Remark \ref{justifyG4}, making use of the joint continuity of the Martin kernel of $(-\Delta)_{\CFL}^s$ on $\Omega \times \partial \Omega$ (see \cite[Theorem 1.2]{ChenKim_2002}) and the continuity of the limit function $G^{\Omega}(x,y)/\delta(x)^{2s-1}$ (see \cite[Remark 3.14]{Aba_2019}). 
\section{Main results and comparison with related works}\label{sec:result}
Before stating our main results, we recall standard definitions of function and measure spaces that will be used in the sequel. For $1 \leq q <\infty$ and $\alpha \in \R$,  let $L^q(\Omega,\delta^\alpha)$ be the weighted Lebesgue space defined by
$$
L^q(\Omega,\delta^{\alpha}) := \left\{ u \in L^1_{\loc}(\Omega): \norm{u}_{L^q(\Omega,\delta^{\alpha})}:= \left( \int_{\Omega}|u|^q \delta^{\alpha}dx \right)^{\frac{1}{q}}< +\infty \right\}
$$
and let
$$\delta^{\alpha} L^{\infty}(\Omega) := \{ \delta^{\alpha}u: u \in L^\infty(\Omega)\}. 
$$
For $\alpha \geq 0$, denote by $\M(\Omega,\delta^{\alpha})$ the space of Radon measures on $\Omega$ such that $\norm{\mu}_{\M(\Omega,\delta^{\alpha})} := \int_{\Omega} \delta^{\alpha} d|\mu| < +\infty$.  This space is also known as the dual space of the space of (weighted) continuous functions
\[ C(\overline{\Omega},\delta^{-\alpha}) := \left\{ f \in C(\overline{\Omega}): f\delta^{-\alpha} \in C(\overline{\Omega}) \right\}.\]
When $\alpha = 0$, we use notation $\M(\Omega)$ for the space of bounded Radon measures on $\Omega$. More details on the given spaces could be found in \cite{Gra_2009,MarVer_2014}.  The cone of positive measures in $\M(\Omega,\delta^{\alpha})$ is denoted by $\M^+(\Omega,\delta^{\alpha})$.
\subsection{Compactness of the Green operator} Let $G^{\Omega}$ be the Green function introduced in \eqref{eq_G2}--\eqref{eq_G3bis}.  For $f \in L^1(\Omega,\delta^{\gamma})$ and $\mu \in \M(\Omega,\delta^{\gamma})$, we define
\begin{align*}
\Gbb^{\Omega}[f](x)&:= \int_{\Omega} G^{\Omega}(x,y)f(y) dy, \quad x \in \Omega,\\
\Gbb^{\Omega}[\mu](x)&:= \int_{\Omega} G^{\Omega}(x,y) d\mu(y), \quad x \in \Omega.
\end{align*}
By assumption \eqref{eq_G3bis},  we know that $\Gbb^{\Omega}[\mu]$ is well-defined.  By \cite[Lemma 4.2]{TruongTai_2020} that $\Gbb^{\Omega}[\mu](x)$ is finite for a.e. $x \in \Omega$ if and only if $\mu \in \M(\Omega,\delta^\gamma)$.


Our first main result concerns compactness properties of $\Gbb^{\Omega}$ in weighted function and measure spaces. To begin, for $\beta \geq 0$ and $\alpha \geq \beta-2s$, we set
\begin{align*} 
p^*_{\beta, \alpha}:=\frac{N+\alpha}{N+\beta-2s}.
\end{align*}
In the case $\alpha=\beta=\gamma$, we simply write $p^*$ instead of $p_{\gamma,\gamma}^*$, namely
\begin{equation} \label{p*}
p^*=\frac{N+\gamma}{N+\gamma-2s}.
\end{equation}

\begin{theorem}[Compactness of $\Gbb^{\Omega}$]\label{theo:compactness} Assume that \eqref{eq_G2}--\eqref{eq_G3bis} hold and let $s,\gamma$ be as in \eqref{G-est}. 
	
$(i)$ Let $\gamma' \in [0,\gamma)$ and $\alpha$ satisfy
$$(-\gamma' - 1)\vee (\gamma' - 2s)\vee \left(- \frac{\gamma' N}{N - 2s + \gamma'}\right) < \alpha < \frac{\gamma' N}{N-2s}.$$
Then the map
$\Gbb^{\Omega}: \M(\Omega, \delta^{\gamma'}) \to L^q(\Omega, \delta^{\alpha})$ is compact for every $q < p^*_{\gamma', \alpha}=\frac{N+\alpha}{N+\gamma'-2s}$.

$(ii)$ In addition, if \eqref{eq_G4} holds, then statement $(i)$ holds for $\gamma' = \gamma$.

\end{theorem}

Let us give some remarks on the above result:

(a) Unlike the approach employed in case of the classical Laplacian \cite{MarVer_2014} or the RFL \cite[Proposition 2.6]{CheVer_2014} which relies essentially on the particular operator's inherent properties and available compact Sobolev embeddings,  our unified method, inspired by \cite[Proposition 5.1]{BonFigVaz_2018}, is based on a profound analysis of the Green function and its subtle regularized versions, making use of  Riesz--Fr\'echet--Kolmogorov theorem.  With properties \eqref{eq_G2}--\eqref{eq_G3bis} of Green functions, the method allows to deal with the space $\M(\Omega,\delta^{\gamma'})$ for any $\gamma' \in [0,\gamma)$.  Condition \eqref{eq_G4} is additionally imposed in statement (ii) to enable us to include also the space $\M(\Omega,\delta^{\gamma})$, which, by virtue of \eqref{G-est}, seems to be the largest weighted measure space for the compactness of $\Gbb^{\Omega}$ to hold.  

(b) Theorem \ref{theo:compactness} covers and extends several results in the literature. Indeed, the fact that the Green operator $\Gbb^{\Omega}$ of the classical Laplacian $-\Delta$ is compactly embedded from $\M(\Omega,\delta)$ to $L^q(\Omega)$ with $p < \frac{N}{N-1}$  (see \cite{MonPon_2008}) can be retrieved from statement (ii) with $\gamma=s=1$ and $\alpha=0$.  Moreover, the compactness from $L^1(\Omega,\delta^s)$ to $L^q(\Omega)$, $q < \frac{N}{N-s}$, of the Green operator associated to the RFL (see \cite[Proposition 2.6]{CheVer_2014}) is a special case of statement (ii) with $\gamma = s \in (0,1)$ and $\alpha=0$. In addition, statement (ii) with $s \in (0,1), \gamma=\alpha=1$ and with $s \in (\frac{1}{2},1), \gamma=2s-1, \alpha = 0$ reduce to a compactness result in the case of SFL and in the case of CFL, respectively.  {The compact property of the Green operators in the two latter cases,  or in general in the weighted setting with irregular data,  seems to be new to the best of our knowledge.}

It is worth mentioning that the value $p^*$ given by \eqref{p*} is the critical exponent for the existence of a solution to semilinear equations with a source power term (see \cite[Theorem 3.3]{TruongTai_2020}). We show below that this exponent also plays an important role in the solvability of problem  \eqref{eq:equation_introduction}--\eqref{eq:bdry-cond}.

\subsection{Semilinear elliptic equations}The above results are important tools in the study of the following problem
\begin{equation}\label{eq:semilinear_absorption}
\left\{ \begin{aligned}
\L u + g(u) &= \mu &&\text{ in }\Omega, \\
u  &= 0 &&\text{ on }\partial \Omega \text{ or in }{\R^N \backslash \Omega} \text{ (if applicable)},
\end{aligned} \right.
\end{equation}
where $\mu \in \M(\Omega,\delta^{\gamma})$ and $g: \R \to \R$ is a nondecreasing continuous function with $g(0)=0$. In order to study \eqref{eq:semilinear_absorption}, we adapt the notion of weak-dual solutions which was first introduced by Bonforte and V\'azquez \cite{BonVaz_2015} and then was used in several works (see e.g. \cite{BonFigVaz_2018,Aba_2019,ChaGomVaz_2019_2020,chan2020singular}).

\begin{definition} Assume that $\mu \in \M(\Omega,\delta^{\gamma})$. A function $u$ is called a \textit{weak-dual solution} of \eqref{eq:semilinear_absorption} if $u \in L^1(\Omega,\delta^{\gamma})$, $g(u) \in  L^1(\Omega,\delta^{\gamma})$ and
\begin{equation} \label{weakdualsol}\int_{\Omega} u\xi dx + \int_{\Omega} g(u)\Gbb^{\Omega}[\xi] dx = \int_{\Omega} \Gbb^{\Omega}[\xi] d\mu,\quad \forall \xi \in \delta^{\gamma}L^{\infty}(\Omega).
\end{equation}
\end{definition}
Here all the integrals in \eqref{weakdualsol} are well-defined under assumptions \eqref{eq_G2}--\eqref{eq_G3bis}. The notion of weak-dual solutions has interesting features. On the one hand, as one can see from formulation \eqref{weakdualsol}, an advantage of this notion in comparison with other concepts of solutions is that it involves only the Green function and does not require to specify the meaning of $\L \xi$ for test functions $\xi$. Moreover, it is equivalent to other notions of solutions given in \cite[Theorem 2.1]{ChaGomVaz_2019_2020} for which the space of test functions $\delta^{\gamma}L^{\infty}(\Omega)$ is replaced by $L^{\infty}_c(\Omega)$ or $C^{\infty}_c(\Omega)$. 
On the other hand, since the operator $\L$ does not appear in \eqref{weakdualsol}, its properties cannot be directly exploited from this formulation, which therefore requires an analysis mainly on the Green function.

Note that $u$ is a weak-dual solution of problem \eqref{eq:semilinear_absorption} if and only if
\begin{equation} \label{representation}
u + \Gbb^{\Omega}[g(u)] = \Gbb^{\Omega}[\mu] \quad \text{a.e. in } \Omega.	
\end{equation}

In light of  Theorem \ref{theo:compactness}, Theorem \ref{thm:kato} and the ideas in \cite{Pon_2016}, we are able to obtain the solvability for \eqref{eq:semilinear_absorption}.

\begin{theorem}\label{theo:subsupersolution} Assume that \eqref{eq_L1}--\eqref{eq_L4} and \eqref{eq_G2}--\eqref{eq_G4} hold and $\mu \in \M(\Omega,\delta^{\gamma})$. Let $g: \R \to \R$ be a nondecreasing continuous function with $g(0)=0$. Assume in addition that
\begin{equation}\label{eq:goodmeasure}
g(-\Gbb^{\Omega}[\mu^-]), g(\Gbb^{\Omega}[\mu^+]) \in L^1(\Omega,\delta^{\gamma}).
\end{equation}
Then problem \eqref{eq:semilinear_absorption} admits a unique weak-dual solution $u$. This solution satisfies
\begin{equation} \label{Gmu+-}-\Gbb^{\Omega}[\mu^-] \le u \le \Gbb^{\Omega}[\mu^+] \quad \text{a.e. in } \Omega.
\end{equation}
Moreover, there exists a positive constant $C=C(N,\Omega,s,\gamma)$ such that
\begin{equation} \label{est:apriori}
	\| u \|_{L^1(\Omega,\delta^\gamma)}	+ \| g(u) \|_{L^1(\Omega,\delta^\gamma)} \leq C \, \| \mu \|_{\M(\Omega,\delta^\gamma)}.
\end{equation}
Furthermore, the map $\mu \mapsto u$ is nondecreasing.

\end{theorem}

When the data belong to $L^1(\Omega,\delta^\gamma)$, condition \eqref{eq:goodmeasure} can be relaxed as pointed out in the following result.
\begin{corollary} \label{cor:L1data} Assume that \eqref{eq_L1}--\eqref{eq_L4} and \eqref{eq_G2}--\eqref{eq_G4} hold. For any $f \in L^1(\Omega,\delta^{\gamma})$, problem \eqref{eq:semilinear_absorption} with $\mu=f$ admits a unique weak-dual solution. Furthermore, the map $f \mapsto u$ is increasing.
\end{corollary}
As a consequence of Theorem \ref{theo:subsupersolution}, the existence and uniqueness remain to hold when $g$ satisfies the so-called \textit{subcritical integral condition}
\begin{equation} \label{sub-int} \int_1^{\infty} [g(t) - g(-t)]t^{-1-p^*}dt < +\infty.
\end{equation}
\begin{theorem} \label{measuredata-sub} Assume that \eqref{eq_L1}--\eqref{eq_L4} and \eqref{eq_G2}--\eqref{eq_G4} hold. Let $g: \R \to \R$ be a nondecreasing continuous function satisfying \eqref{sub-int} and $g(0)=0$. Then for any $\mu \in \M(\Omega,\delta^{\gamma})$, problem \eqref{eq:semilinear_absorption} admits a unique weak-dual solution $u$. Furthermore, the solution satisfies \eqref{Gmu+-}, \eqref{est:apriori} and the map $\mu \mapsto u$ is nondecreasing.
\end{theorem}

Next we focus on the case where $g$ is a power function, namely $g(t)=|t|^{p-1}t$ with $p>1$. 
In this case, problem \eqref{eq:semilinear_absorption} becomes
\begin{equation}\label{eq:up}
	\left\{ \begin{aligned}
		\L u + |u|^{p-1}u &= \mu &&\text{ in }\Omega, \\
		u  &= 0 &&\text{ on }\partial \Omega \text{ or in }\R^N \backslash \Omega \text{ (if applicable)},
	\end{aligned} \right.
\end{equation}
and condition \eqref{sub-int} holds if and only if $1<p<p^*$.
By Theorem \ref{measuredata-sub}, if $1<p<p^*$ then problem \eqref{eq:up} has a unique weak-dual solution. 

The case $p \geq p^*$ is more challenging and requires a profound analysis in a suitable capacitary setting. We will make use of  Bessel capacities (see the definition in Section \ref{sec:sufficientcondition}).

We say that a measure $\nu$ in $\Omega$ is \textit{absolutely continuous} with respect to the  capacity $\mathrm{Cap}_{\alpha,q}$ if
$$ \forall E \subset \Omega, \, E \text{ is Borel}, \quad  \mathrm{Cap}_{\alpha,q}(E) = 0 \Longrightarrow |\nu|(E) = 0.
$$ 

The next theorem provides a sufficient condition expressed in terms of a suitable Bessel capacity under which problem \eqref{eq:up} has a (unique) weak-dual solution.  
\begin{theorem} \label{thm:supercritical}
Assume that \eqref{eq_L1}--\eqref{eq_L4} and \eqref{eq_G2}--\eqref{eq_G4} hold. Let $\mu \in \M(\Omega,\delta^\gamma)$,  $d\nu = \delta^\gamma d\mu$ and $p \in (1,\infty)$ such that $s-\frac{\gamma}{2p'}>0$. If $\nu$ is absolutely continuous with respect to the Bessel capacity $\mathrm{Cap}_{2s-\frac{\gamma}{p'},p'}$ then problem \eqref{eq:up} has a unique weak-dual solution. Moreover the solution satisfies \eqref{Gmu+-} and \eqref{est:apriori}.  
\end{theorem}

The characterization of measures in terms of Bessel capacities for the solvability of nonlinear elliptic equations has been first obtained by Baras and Pierre in \cite{BarPie-1984} for the case of Laplacian ($s = 1$ and $\gamma = 1$), and has been obtained by Chen and V\'eron \cite{CheVer_2014} for the case of RFL ($s=\gamma \in (0,1)$).  
Theorem \ref{thm:supercritical} extends the sufficient condition established in the mentioned cases to a weighted measure space, and seems to be new in case of SFL with $s > \frac{1}{2}$ and CFL. It is not known if the absolute continuity of $\nu$ is also a necessary condition for the solvability of the problem \eqref{eq:up}. Since the techniques in \cite{BarPie-1984,CheVer_2014} cannot be adapted to the context of weak-dual solutions, it entails a different approach. 

It is worth noting that when $1<p<p^*$, the only set with zero $\mathrm{Cap}_{2s-\frac{\gamma}{p'},p'}$-capacity is the empty set, hence by Theorem \ref{thm:supercritical}, problem \eqref{eq:up} has a unique weak-dual solution.

We close this section with a remark on the existence of a solution to \eqref{eq:equation_introduction} with an isolated boundary singularity which is a by-product. Let $z \in \partial \Omega$ and $\{z_n\}_{n \in \N} \subset \Omega$ converging to $z$. When $1<p<p^*$, for any $n \in \N$, there exists a unique weak-dual solution $u_n$ to \eqref{eq:up} with $\mu=\delta^{-\gamma}\delta_{z_n}$, where $\delta_{z_n}$ denotes the Dirac measure concentrated at $z_n$. In Appendix \ref{sec:boundarysolution}, we will use the compactness property of the Green operator to show that when $n \to \infty$, the corresponding solution $u_n$ converges to a function $u$ that is singular at $z$ and has the same blow-up rate at $z$ as $M^\Omega(\cdot,z)$, where $M^\Omega(\cdot,z)$ is given in \eqref{eq:martinkernel_def} (see Proposition \ref{bdry-iso}).

\section{Compactness of the Green operator}\label{sec:lineartheory}

Let us start with the definition of weighted Lebesgue spaces. 
The weighted Marcinkiewicz spaces, or weak-Lebesgue spaces, $M^q(\Omega,\delta^{\alpha})$, $1 \le q < \infty$,  are defined by
$$
M^q(\Omega,\delta^{\alpha}):= \left\{ u \in L^1_{\loc}(\Omega):  \sup_{\lambda > 0} \lambda^q \int_{\Omega} \1_{\{ x \in \Omega: |u(x)| > \lambda \} } \delta^{\alpha} dx < +\infty \right\},
$$
where $\1_{E}$ denotes the indicator function of a set $E \subset \R^N$. Put
$$
\vertiii{u}_{M^q(\Omega,\delta^\alpha)}:= \left( \sup_{\lambda > 0} \lambda^q \int_{\Omega} \1_{\{ x \in \Omega: |u(x)| > \lambda \} } \delta^{\alpha} dx \right)^{\frac{1}{q} }.
$$
We note that $\vertiii{\cdot}_{M^q(\Omega,\delta^\alpha)}$ is not a norm, but for $q >1$, it is equivalent to the norm
$$
\norm{u}_{M^q(\Omega,\delta^\alpha)}:= \sup \left\{ \frac{\int_{A} |u|^q \delta^\alpha dx }{ (\int_{A} \delta^\alpha dx)^{1-\frac{1}{q}} }: A \subset \Omega, A \text{ measurable }, 0< \int_{A} \delta^\alpha dx < + \infty. \right\}.
$$
In fact, there hold (see \cite[page 300-311]{CiCo})
\begin{equation} \label{Marcin-equi}
	\vertiii{u}_{M^q(\Omega,\delta^\alpha)} \leq \norm{u}_{M^q(\Omega,\delta^\alpha)} \leq \frac{q}{q-1}\vertiii{u}_{M^q(\Omega,\delta^\alpha)}, \quad \forall u \in M^q(\Omega,\delta^\alpha).
\end{equation}

It is well-known that the following embeddings hold
$$L^q(\Omega,\delta^{\alpha}) \subset M^q(\Omega,\delta^{\alpha}) \subset L^{r}(\Omega,\delta^{\alpha}),\quad \forall r \in [1,q),$$
where all the inclusions are strict and continuous, see for instance \cite[Section 2.2]{BidViv_2000} and \cite[Section 1.1]{Gra_2009}. We recover the usual $L^q$ spaces and weak-$L^q$ spaces if $\alpha = 0$ and simply write $L^q(\Omega)$ and $M^q(\Omega)$ respectively. 

The following result shows the continuity of $\Gbb^{\Omega}$.

\begin{proposition}\label{prop:marcinestimate} Assume that \eqref{eq_G2}--\eqref{eq_G3bis} hold. Let $\gamma' \in [0,\gamma]$ and $\alpha$ satisfy
\begin{align*}
\left(-\gamma' - 1\right) \vee \left(\gamma' - 2s\right) \vee \left(- \frac{\gamma' N}{N - 2s + \gamma'}\right) < \alpha < \frac{\gamma' N}{N-2s}.
\end{align*}
Then the map
\begin{equation}\label{eq:Marcin}
\Gbb^{\Omega}: \M(\Omega,\delta^{\gamma'}) \to M^{p_{\gamma', \alpha}^*}(\Omega,\delta^{\alpha})
\end{equation}
is continuous (recall that $p_{\gamma',\alpha}^*=\frac{N+\alpha}{N+\gamma'-2s}$).
In particular, for any $q \in [1,p_{\gamma',\alpha})$ the map
\begin{equation}\label{eq:Marcin2}
\Gbb^{\Omega}: \M(\Omega,\delta^{\gamma'}) \to L^q(\Omega,\delta^{\alpha})
\end{equation}
is continuous.
\end{proposition}

The desired result can be obtained by proceeding as in the proof of \cite[Proposition 4.6]{TruongTai_2020}, replacing  $\gamma$ by $\gamma'$. Therefore we omit the proof.

In order to prove Theorem \ref{theo:compactness}, we need the following technical result.


\begin{lemma}\label{lem:technical}
Assume that \eqref{eq_G2}--\eqref{eq_G3bis} hold.
For $\varepsilon > 0$ and $\beta > N-2s+2\gamma$, we define $K^{\varepsilon}_{\beta} : \R^+\to \R^+$ as 
\begin{equation} \label{Keb}
K^{\varepsilon}_{\beta} (t) :=1 \wedge \left(t^{\beta}\varepsilon^{-\beta}\right), \quad t \in \R^+.
\end{equation}
Put
\begin{equation} \label{eq:Gepsilon} 
G^{\varepsilon}_{\beta}(x,y) :=  \left\{ \begin{aligned}
&G^{\Omega}(x,y)K^{\varepsilon}_{\beta}(|x-y|) \quad &&\text{if } (x,y) \in (\Omega \times \Omega) \backslash D_{\Omega},\\
&0 \quad &&\text{if } (x,y) \in D_{\Omega},
\end{aligned}\right.
\end{equation}
where $G^{\Omega}$ is the Green function of $\L$.

(i) For every $\gamma' < \gamma$, the map
$$(x,y) \mapsto \dfrac{G^{\varepsilon}_{\beta}(x,y)}{\delta(y)^{\gamma'}}$$
is uniformly continuous in $\Omega \times \Omega$. 

(ii) If, in addition, \eqref{eq_G4} holds, then the map
$$(x,y) \mapsto \dfrac{G^{\varepsilon}_{\beta}(x,y)}{\delta(y)^{\gamma}}$$
is uniformly continuous in $\Omega \times \Omega$.
\end{lemma}

The proof of Lemma \ref{lem:technical} is given in Appendix~\ref{appendix:proof_of_lemma}.

\begin{proof}[{\sc Proof of Theorem \ref{theo:compactness}}] \text{}

(i) \textbf{Step 1.} Let $\varepsilon > 0$ and $G^{\varepsilon}_{\beta}$, $K^{\varepsilon}_{\beta}$ defined in Lemma \ref{lem:technical}. For simplicity, we write
$G^{\varepsilon}$ for $G^{\varepsilon}_{\beta}$, $K^\varepsilon$ for $K_\beta^\varepsilon$ and define $H^{\varepsilon}:= G^{\Omega} - G^{\varepsilon}$.
For every $0 \le \gamma' < \gamma$,  by \eqref{G-est} and the fact that $a^{\gamma'} \wedge 1 \le a^{\gamma}$
\begin{equation}\label{eq:Green0}
G^{\Omega}(x,y) \lesssim \left( \dfrac{\delta(y)^{\gamma}}{\delta(x)^{\gamma}} \wedge 1 \right)\dfrac{1}{|x-y|^{N-2s}}  \lesssim  \dfrac{\delta(y)^{\gamma'}}{\delta(x)^{\gamma'}} \cdot \dfrac{1}{|x-y|^{N-2s}},
\end{equation}
which implies
\begin{equation}\label{eq:hcompactL1}
0 \le H^{\varepsilon}(x,y) = G^{\Omega}(x,y)(1- K^{\varepsilon}(x,y)) \lesssim \dfrac{\delta(y)^{\gamma'}}{\delta(x)^{\gamma'}} b(|x-y|), \quad x,y \in \Omega,
\end{equation}
where 
$$
b(t) := \dfrac{1}{t^{N-2s}} \left(1 - \dfrac{t^\beta}{\varepsilon^{\beta}}\right)\1_{\{ 0 \le t \le \varepsilon \} }, \quad t \geq 0.
$$
Let $K \subset \subset \Omega$ be fixed and $\mu \in \M(\Omega,\delta^{\gamma'})$.  
For $h \in \R^N$ such that $|h| < \frac{1}{2}\dist(K,\partial \Omega)$, one has
\begin{equation}\label{eq_maintheo2}
\begin{aligned}
\norm{ \Gbb^{\Omega}[\mu](\cdot + h) - \Gbb^{\Omega}[\mu](\cdot) }_{L^1(K)} &= \norm{ \int_{\Omega} \left[G^{\Omega}(\cdot+h,y) - G^{\Omega}(\cdot,y)\right]d\mu(y) }_{L^1(K)} \\ &\le \norm{ \int_{\Omega} H^{\varepsilon} (\cdot,y)d\mu(y) }_{L^1(K)} + \norm{ \int_{\Omega} H^{\varepsilon} (\cdot+h,y)d\mu(y) }_{L^1(K)} \\ &+ \norm{ \int_{\Omega} \left[G^{\varepsilon} (\cdot+h,y) - G^{\varepsilon}(\cdot,y) \right]d\mu(y)}_{L^1(K)} \\
&=: J_1 + J_2 + J_3.
\end{aligned}
\end{equation}

We first estimate $J_1$.  Using  the inequality  $\delta(x) \geq \dist(K,\partial \Omega),\forall x \in K$ one has
$$\begin{aligned}
J_1 = \int_{K} \int_{\Omega} H^{\varepsilon} (x,y) d\mu(y) \, dx &\lesssim \int_{K}\int_{\Omega} \dfrac{\delta(y)^{\gamma'}}{\delta(x)^{\gamma'}} b(|x-y|) d\mu(y) \, dx \\
& \lesssim \int_{\Omega}\int_{K} \dfrac{\delta(y)^{\gamma'}}{\delta(x)^{\gamma'}} b(|x-y|) dx \, d\mu(y)\\
& \leq \dfrac{1}{\dist(K,\partial \Omega)^{\gamma'}} \int_{\Omega}\delta(y)^{\gamma'} \left(\int_{K}  b(|x-y|) dx\right)\, d\mu(y).
\end{aligned}
$$
Since 
$$\int_{K} b(|x-y|) dx = \int_{ \{|x- y| \le \varepsilon\} } b(|x-y|)dx  \lesssim \int_0^{\varepsilon} b(t)t^{N-1} dt \lesssim \int_0^{\varepsilon} t^{2s-1} dt \lesssim \varepsilon^{2s},$$
we obtain
\begin{equation}\label{eq_I12}
\begin{aligned}
J_1 \lesssim  \dfrac{1}{\dist(K,\partial \Omega)^{\gamma'}}\varepsilon^{2s} \int_{\Omega} \delta(y)^{\gamma'}  d\mu(y) = \dfrac{1}{\dist(K,\partial \Omega)^{\gamma'}}\varepsilon^{2s} \norm{\mu}_{\M(\Omega,\delta^{\gamma'})}.
\end{aligned}
\end{equation}

Next we estimate $J_2$. We have
\begin{align*}
J_2 = \int_{K} \int_{\Omega} H^{\varepsilon} (x+h,y) d\mu(y) dx 
&\lesssim \int_{K}\int_{\Omega} \dfrac{\delta(y)^{\gamma'}}{\delta(x+h)^{\gamma'}} b(|x+h-y|) d\mu(y) dx \\
&\lesssim \int_{\Omega} \delta(y)^{\gamma'} \left( \int_K \dfrac{b(|x+h-y|)}{\delta(x+h)^{\gamma'}}   dx \right) d\mu(y).
\end{align*}
Notice that for $|h| < \frac{1}{2}\dist(K,\partial \Omega)$,
 $$\delta(x+h)^{\gamma'} \geq \delta(x)^{\gamma'} - |h|^{\gamma'} \geq C_{\gamma'} \dist(K,\partial \Omega)^{\gamma'},\quad \forall x \in K,$$
which implies
\begin{align*}
\int_{K} \dfrac{b(|x+h-y| )}{\delta(x+h)^{\gamma'}} dx 
& \lesssim \frac{1}{C_{\gamma'}\dist(K,\partial \Omega)^{\gamma'}} \int_{K} b(|x+h-y|) dx \\
& \lesssim \frac{1}{C_{\gamma'}\dist(K,\partial \Omega)^{\gamma'}} \int_{\{|x+h-y| \le \varepsilon \}} b(|x+h-y|) dx\\ 
& \lesssim \frac{1}{C_{\gamma'}\dist(K,\partial \Omega)^{\gamma'}} \int_0^{\varepsilon} b(t)t^{N-1} dt \lesssim \varepsilon^{2s}.
\end{align*}
From this we derive that
\begin{equation}\label{eq_I22}
\begin{aligned}
J_2 \lesssim  \frac{1}{C_{\gamma'}\dist(K,\partial \Omega)^{\gamma'}}\varepsilon^{2s} \int_{\Omega} \delta(y)^{\gamma'}  d\mu(y) = \frac{1}{C_{\gamma'}\dist(K,\partial \Omega)^{\gamma'}}\varepsilon^{2s} \norm{\mu}_{\M(\Omega,\delta^{\gamma'})}.
\end{aligned}
\end{equation}

Finally,  for $J_3$ we write
\begin{equation}\label{eq_I23}
\begin{aligned}
J_3 &= \int_{K} \left|\int_{\Omega} \left[G^{\varepsilon} (x+h,y) - G^{\varepsilon}(x,y) \right]d\mu(y)\right|dx\\
& \le \int_{\Omega} \int_{K} |G^{\varepsilon}(x+h,y) - G^{\varepsilon}(x,y)| dx d\mu(y)\\
&\le \int_{\Omega} \delta(y)^{\gamma'} \int_{K} \left|\dfrac{G^{\varepsilon}(x+h,y) - G^{\varepsilon}(x,y)}{\delta(y)^{\gamma'}}\right| dx d\mu(y)\\
&\le |K| \norm{ \dfrac{G^{\varepsilon}(\cdot + h, \cdot)}{\delta(\cdot)^{\gamma'}} - \dfrac{G^{\varepsilon}(\cdot, \cdot)}{\delta(\cdot)^{\gamma'}}}_{L^{\infty}(\Omega \times \Omega)} \norm{\mu}_{\M(\Omega,\delta^{\gamma'})}.
\end{aligned}
\end{equation}
Since the map $(x,y) \mapsto \frac{G^{\varepsilon}(x,y)}{\delta(y)^{\gamma'}}$ is uniformly continuous by Lemma \ref{lem:technical}(i),  we deduce that
\begin{equation}\label{eq_I23.2}
\lim_{|h| \to 0} \sup_{ \norm{\mu}_{\M(\Omega,\delta^{\gamma'})} \le 1}   J_3 = 0.
\end{equation}
Combining \eqref{eq_maintheo2}--\eqref{eq_I23.2},  we have
$$\limsup_{|h| \to 0}\sup_{ \norm{\mu}_{\M(\Omega,\delta^{\gamma'})} \le 1} \norm{ \Gbb^{\Omega}[\mu](\cdot+h) - \Gbb^{\Omega}[\mu](\cdot)}_{L^1(K)} \le C(K)\varepsilon^{2s}.$$
By Riesz--Fr\'echet--Kolmogorov theorem  (see e.g. \cite[Proposition 1.2.23]{DraMil_2013}) we conclude that $\Gbb^{\Omega}: \M(\Omega,\delta^{\gamma'}) \to L^1(K)$ is compact for every compact set $K \subset \subset \Omega$.

\textbf{Step 2.} We prove that  $\Gbb^{\Omega}: \M(\Omega,\delta^{\gamma'}) \to L^q(\Omega,\delta^{\alpha})$ is compact.  Consider a bounded sequence $\{ \mu_n\}_{n \in \mathbb{N}} \subset \M(\Omega,\delta^{\gamma'})$ and a  $C^2$ exhaustion $\{ \Omega_k \}_{k \in \mathbb{N}}$ of $\Omega$, namely $\Omega_k \subset \subset \Omega_{k+1} \subset \subset \Omega$ for all $k \in \mathbb{N}$ and $\cup_{k=1}^{\infty} \Omega_k = \Omega$. Put $u_n:= \Gbb^{\Omega}[\mu_n]$.  By Step 1, for any $k \in \N$, the map $\Gbb^{\Omega} : \M(\Omega,\delta^{\gamma'}) \to L^1(\Omega_k)$ is compact. Thus, there exists a subsequence, denoted by $\{u_{n,k}\}_{n \in \mathbb{N}}$, and a function $v_k$ defined in $\Omega_k$ such that $u_{n,k} \to v_k$ a.e. in $\Omega_k$ as $n \to \infty$. Using the standard diagonal argument, we derive that  there exist a subsequence, still labeled by the same notation $\{u_n\}_{n \in \mathbb{N}}$, and  a function $u$ such that $u_n \to u$ a.e. in $\Omega$ and $u = v_k$ in $\Omega_k$. On the other hand, by \eqref{eq:Marcin2}, for any $q \in [1, p^*_{\gamma',\alpha} )$, we have
$$
\norm{u_n}_{L^q(\Omega,\delta^{\alpha})} \lesssim \norm{\mu_n}_{\M(\Omega,\delta^{\gamma'})} \lesssim 1.
$$
Employing Vitali's convergence theorem, we derive that, up to a subsequence, $u_n \to u$ in $L^q(\Omega,\delta^{\alpha})$ for any $q \in [1, p^*_{\gamma',\alpha})$.  Hence, we conclude that the map $\Gbb^{\Omega}: \M(\Omega,\delta^{\gamma'}) \to L^q(\Omega,\delta^{\alpha})$ is compact for any $q \in [1,p^*_{{\gamma'},\alpha})$.

(ii) Let $\mu \in \M(\Omega,\delta^{\gamma})$ and consider the decomposition as in \eqref{eq_maintheo2} for $\Gbb^{\Omega}[\mu]$.  Following the same argument as above,  we have the same estimates for $J_1, J_2$ where $\gamma'$ is replaced by $\gamma$.  Furthermore,  the estimate for $J_3$ also holds under assumption \eqref{eq_G4} by Lemma \ref{lem:technical}(ii). Hence,  the map $\Gbb^{\Omega}: \M(\Omega,\delta^{\gamma}) \to L^1(K)$ is compact for any subset $K \subset \subset \Omega$.  Finally, following Step 2 in the proof of (i),  we conclude that the map $\Gbb^{\Omega}: \M(\Omega,\delta^{\gamma}) \to L^q(\Omega,\delta^{\alpha})$ is compact for every $q \in [1, {p^*_{\gamma,\alpha})}$ (recall that $p^*_{\gamma,\alpha}=\frac{N+\alpha}{N+\gamma-2s}$).
\end{proof}

In particular, we infer from Theorem \ref{theo:compactness} that the map $\Gbb^{\Omega}: \M(\Omega,\delta^{\gamma}) \to L^1(\Omega,\delta^{\gamma})$ is compact under the given assumptions.  In addition,  as a consequence of Theorem \ref{theo:compactness}(ii), we obtain the following convergence which can be interpreted as the stability of solutions to linear problems.

\begin{corollary} \label{cor:stability} Suppose that \eqref{eq_G2}--\eqref{eq_G4} hold. Assume $\{ \mu_n\}_{n \in \N} \subset \M(\Omega,\delta^{\gamma})$ converges weakly to $\mu$ in $\M(\Omega,\delta^{\gamma})$.  Then  $\Gbb^{\Omega}[\mu_n] \to \Gbb^{\Omega}[\mu]$ in $L^1(\Omega,\delta^{\gamma})$.
\end{corollary}

\begin{proof} Let $\{ \nu_n\}_{n \in \mathbb{N}}$ be an arbitrary subsequence of $\{\mu_n\}_{n \in \mathbb{N}}$. Put $u_{\nu_n} := \Gbb^{\Omega}[\nu_n], n \in \N$, and $u_\mu := \Gbb^{\Omega}[\mu]$. By the compactness of the map $\Gbb^{\Omega}: \M(\Omega,\delta^{\gamma}) \to L^1(\Omega,\delta^{\gamma})$ in Theorem \ref{theo:compactness} (ii), up to a subsequence, there exists a function $v \in L^1(\Omega,\delta^{\gamma})$ such that $u_{\nu_n} \to v$ in $L^1(\Omega,\delta^{\gamma})$. By the integration-by-parts formula \cite[Lemma 4.4]{TruongTai_2020}, we have 
$$\int_{\Omega} u_{\nu_n} \xi dx = \int_{\Omega} \Gbb^{\Omega} [\xi] d\nu_n,\quad \forall \xi \in \delta^{\gamma} L^{\infty}(\Omega).$$
Letting $n \to \infty$ and noticing  that $\{\nu_n\}_{n \in \mathbb{N}}$ converges weakly to $\mu$, we have
$$\int_{\Omega} v \xi dx = \int_{\Omega} \Gbb^{\Omega} [\xi] d\mu,\quad \forall \xi \in \delta^{\gamma} L^{\infty}(\Omega),$$
which gives $v = \Gbb^{\Omega}[\mu] = u_\mu$.  Hence, $u_{\nu_n} \to u_\mu$ in $L^1(\Omega,\delta^{\gamma})$. Since $\{ \nu_n\}_{n \in \mathbb{N}}$ is chosen arbitrarily, we conclude that the whole sequence $\{\Gbb^{\Omega}[\mu_n]\}_{n \in \mathbb{N}}$ converges to $\Gbb^{\Omega}[\mu]$ in $L^1(\Omega,\delta^{\gamma})$.
\end{proof}

To end this section,  we prove the compactness of the Green operator in $L^r(\Omega,\delta^{\gamma})$ for $r > \frac{N+\gamma}{2s}$ for the sake of completeness.
\begin{proposition} \label{compact-large}
Assume that \eqref{eq_G2}--\eqref{eq_G3bis} hold.  Then the map $\Gbb^{\Omega}: L^r(\Omega,\delta^{\gamma}) \to C(\overline{\Omega})$ is compact {\color{darkblue}for any $r > \frac{N+\gamma}{2s}$}.
\end{proposition}

\begin{proof}

\noindent \textbf{Step 1.} We first claim that the map $\Gbb^{\Omega}\colon L^r(\Omega,\delta^{\gamma}) \to C(\overline{\Omega})$ is continuous.  Firstly,  under  assumption \eqref{eq_G2}, the map $\Gbb^{\Omega}: L^r(\Omega,\delta^{\gamma}) \to L^{\infty}(\Omega)$ is continuous by \cite[Proposition 4.11]{TruongTai_2020}. Hence, it is sufficient to show that $\Gbb^{\Omega}[f]$ is continuous on $\overline{\Omega}$ for any $f \in L^r(\Omega,\delta^\gamma)$. Indeed, let  $\{x_n\}_{n \in \mathbb{N}} \subset \Omega$ be a sequence converging to a point $\tilde x\in \overline \Omega$. 

First we consider the case that $\tilde x \in \Omega$. By assumption \eqref{eq_G3bis},  it can be seen that
\begin{equation}\label{eq:Greenpointwiselim}
\lim_{n \to \infty}\frac{G^{\Omega}(x_n,y)}{\delta(y)^{\gamma}} = \frac{G^{\Omega}(\tilde x,y)}{\delta(y)^{\gamma}} \quad \text{ for a.e. } y \in \Omega.
\end{equation}
	
We show that for every $q \in [1,p^*)$, the sequence $\left\{ G^{\Omega}(x_n ,\cdot)\delta(\cdot)^{-\gamma}\right\}_{n \in \mathbb{N}}$ is bounded in $L^q(\Omega,\delta^{\gamma})$.  For $q \in [1,p^*)$, by \eqref{G-est} there holds
	$$\dfrac{G^{\Omega}(x_n,y)}{\delta(y)^{\gamma(1-\frac{1}{q})}} \lesssim \dfrac{1}{|x_n-y|^{N-2s+\gamma(1-\frac{1}{q})}}, \quad  x_n \neq y.$$
	Therefore, taking into account that $q < p^* = \frac{N+\gamma}{N+\gamma-2s}$, we deduce
	\begin{align*}
		\int_{\Omega} \left( \dfrac{G^{\Omega}(x_n,y)}{\delta(y)^{\gamma}}\right)^q \delta(y)^\gamma dy = \int_{\Omega} \left(\dfrac{G^{\Omega}(x_n,y)}{\delta(y)^{\gamma(1-\frac{1}{q})}}\right)^q dy \lesssim  \int_{\Omega} \dfrac{1}{|x_n-y|^{q(N+\gamma-2s)-\gamma}}dy \leq C(N,s,\gamma,q).
	\end{align*} 
	Thus,  $\left\{ G^{\Omega}(x_n ,\cdot)\delta(\cdot)^{-\gamma}\right\}_{n \in \mathbb{N}}$ is bounded $L^q(\Omega,\delta^{\gamma})$ for every $q \in [1,p^*)$.  Since $r'=\frac{r}{r-1}<p^*$,  we conclude by Vitali's convergence theorem that $G^{\Omega}(x_n,\cdot)\delta(\cdot)^{-\gamma} \to G^{\Omega}(\tilde x,\cdot)\delta(\cdot)^{-\gamma}$ in $L^{r'}(\Omega,\delta^{\gamma})$, namely 
	\begin{equation} \label{eq:compactLinfty.2}
		\lim_{ n \to \infty} \norm{G^{\Omega}(x_n,\cdot)\delta(\cdot)^{-\gamma} - G^{\Omega}(\tilde x,\cdot)\delta(\cdot)^{-\gamma}}_{L^{r'}(\Omega,\delta^{\gamma})} = 0.
	\end{equation}
Let $f \in L^r(\Omega,\delta^{\gamma})$. We have
	\begin{equation}\label{eq:compactLinfty.1}
		\begin{aligned}
			\left|\Gbb^{\Omega}[f](x_n) - \Gbb^{\Omega}[f](\tilde x)\right| &= \left|\int_{\Omega} \left[G^{\Omega}(x_n,y)-G^{\Omega}(\tilde x,y)\right] f(y)dy\right| \\
			&\le \left(\int_{\Omega} f(y)^r \delta(y)^{\gamma} dy\right)^{\frac{1}{r}}\left( \int_{\Omega} \dfrac{|G^{\Omega}(x_n,y)-G^{\Omega}(\tilde x ,y)|^{r'}}{\delta(y)^{\frac{\gamma r'}{r}}}dy\right)^{\frac{1}{r'}}\\
			& = \norm{f}_{L^r(\Omega,\delta^{\gamma})}\norm{G^{\Omega}(x_n,\cdot)\delta(\cdot)^{-\gamma} - G^{\Omega}(\tilde x,\cdot)\delta(\cdot)^{-\gamma}}_{L^{r'}(\Omega,\delta^{\gamma})}.
		\end{aligned}
	\end{equation}
 Using \eqref{eq:compactLinfty.2}, we derive that $\Gbb^{\Omega}[f](x_n) \to \Gbb^{\Omega}[f](\tilde x)$ as $n \to \infty$. Thus $\Gbb^{\Omega}[f]$ is continuous in $\Omega$.
 
 If $\tilde x \in \partial \Omega$ then by using \eqref{G-est}, we deduce that
 $$ \lim_{n \to \infty}\frac{G^{\Omega}(x_n,y)}{\delta(y)^\gamma} = 0 \quad \text{for a.e. } y \in \Omega.
 $$
 By using a similar argument as above, we deduce that for any $f \in L^r(\Omega,\delta^\gamma)$, $\Gbb^\Omega[f](x_n) \to 0$ as $n \to \infty$. Thus, by putting $\Gbb^\Omega[f](x)=0$ for $x \in \partial \Omega$, we obtain that $\Gbb^{\Omega}[f] \in C(\overline \Omega)$.

\textbf{Step 2.} Consider a bounded set $\mathcal{Q} \subset L^r(\Omega,\delta^{\gamma})$ and put $M_{\mathcal{Q}}:=\sup_{f \in \mathcal{Q}}\norm{f}_{L^r(\Omega,\delta^{\gamma})}<+\infty$.  By Step 1 and \cite[Proposition 4.11]{TruongTai_2020},  one has $\Gbb^{\Omega}(\mathcal{Q}) \subset C(\overline{\Omega})$ and 
$$\norm{\Gbb^{\Omega}[f]}_{L^{\infty}(\Omega)} \lesssim \norm{f}_{L^r(\Omega,\delta^{\gamma})} \leq M_\mathcal{Q},\quad \forall f \in \mathcal{Q}.$$
Next let $\varepsilon>0$ and take arbitrary $\tilde x \in \overline \Omega$ and $f \in \mathcal{Q}$. By using the same argument leading to \eqref{eq:compactLinfty.1}, we deduce that there exists $\rho$ depending on $\tilde x, \varepsilon, M_{\mathcal{Q}},\Omega$ such that for any $x \in B(\tilde x,\rho)\cap \Omega$, there holds
$$ \left|\Gbb^{\Omega}[f](x) - \Gbb^{\Omega}[f](\tilde x)\right| < \varepsilon.
$$
It means that the set $\{ \Gbb^{\Omega}[f]: f\in \mathcal{Q} \}$ is equicontinuous.  Invoking Arzel\`a--Ascoli theorem (see for instance \cite[Theorem 1.2.13]{DraMil_2013}), we conclude that the set $\{ \Gbb^{\Omega}[f]: f\in \mathcal{Q} \}$ is relatively compact. Thus $\Gbb^{\Omega}: L^r(\Omega,\delta^{\gamma}) \to C(\overline{\Omega})$ is compact. The proof is complete.
\end{proof}

\begin{remark} The compactness of $\Gbb^{\Omega}$ from $L^2(\Omega)$ into $L^2(\Omega)$ was obtained in \cite{BonFigVaz_2018} by using the Riesz-Fr\'echet-Kolmogorov Theorem. An alternative way prove this property consists of employing compact embedding $\Hbb(\Omega) \hookrightarrow \hookrightarrow L^2(\Omega)$. The question of compactness involving relevant spaces was asked in \cite{ChaGomVaz_2019_2020} and will be discussed in Appendix \ref{appendix:compact}. 
\end{remark}

\section{Kato's inequality} \label{sec:katoinequality}

In this section, we prove Kato's inequality for operators satisfying \eqref{eq_L1}--\eqref{eq_L4} and \eqref{eq_G2}--\eqref{eq_G4}.  We begin by presenting some properties of the function spaces introduced in Section \ref{sec:preliminaries}, which will be used in the sequel.

\begin{lemma}\label{lem:Hs00} 
(i) Assume that $u \in \Hbb(\Omega)$ and $h: \R \to \R$ is a Lipschitz function such that $h(0) = 0$. Then $h(u) \in \Hbb(\Omega)$.

(ii) Assume $u,v \in \Hbb(\Omega) \cap L^{\infty}(\Omega)$. Then $uv \in \Hbb(\Omega)$.	
\end{lemma}

\begin{proof} (i) The desired result follows straightforward from the definition of the $\Hbb(\Omega)$-norm in \eqref{H00-norm} and the Lipschitz properties of $h$.  

(ii) The desired result follows from the definition of the $\Hbb(\Omega)$-norm in \eqref{H00-norm} and the estimate
for any $x,y \in \Omega$,
\begin{align*}
	|u(x)v(x) - u(y)v(y)| \leq \norm{u}_{L^{\infty}(\Omega)} |v(x) - v(y)| + \norm{v}_{L^{\infty}(\Omega)} |u(x) - u(y)|, \quad x,y \in \Omega.
\end{align*}	
The detail of the proof is left to the reader. 	
\end{proof}

\begin{proposition}\label{prop:integrationbypart} Assume that \eqref{eq_L1} and \eqref{eq_L2bis} hold.  Then for every $u, v\in C^{\infty}_c(\Omega)$,
	\begin{equation}\label{eq:integrationbyparts}
		\int_{\Omega} v \L u dx  = \dfrac{1}{2} \int_{\Omega}\int_{\Omega} J(x,y)(u(x)-u(y))(v(x)-v(y)) dx dy + \int_{\Omega} B(x) u(x) v(x) dx.
	\end{equation}
\end{proposition}

\begin{proof} Let $u,v \in C^{\infty}_c(\Omega)$. We first notice that the right-hand side of \eqref{eq:integrationbyparts}
	is finite.  Indeed, let $K \subset \subset \Omega$ such that $\supp u, \supp v \subset \subset K \subset \subset \Omega$. Then by assumption \eqref{eq_L1}, we have
	\begin{align*}
		&\dfrac{1}{2}\int_{\Omega} \int_{\Omega}\left|J(x,y)(u(x)-u(y))(v(x)-v(y))\right|dx dy+ \int_{\Omega}\left|B(x)u(x)v(x)\right| dx\\
		\le \,\, & \dfrac{1}{2} \norm{\nabla u}_{L^{\infty}(\Omega)}\norm{\nabla v}_{L^{\infty}(\Omega)} \int_{\Omega} \int_{\Omega} J(x,y) |x-y|^2dx dy + \norm{u}_{L^{\infty}(\Omega)}\norm{ v}_{L^{\infty}(\Omega)} \int_{K} B(x) dx < +\infty.
	\end{align*}
	
Let $\varepsilon > 0$ be sufficiently small. By \eqref{Lepsilon}, we obtain
	\begin{align*}
		\int_{\Omega} v(x)\L_{\varepsilon} u(x) dx 
		&= \int_{\Omega} v(x)\left[\int_{\Omega} J(x,y)(u(x)-u(y))\chi_{\varepsilon}(|x-y|) dy + B(x)u(x)\right] dx.
	\end{align*}
	Using the symmetry property of $J(x,y)$ and by the change of variables, we get
	\begin{align*}
		&\int_{\Omega} v(x)\left[\int_{\Omega} J(x,y)(u(x)-u(y))\chi_{\varepsilon}(|x-y|) dy \right] dx  \\
		&\quad = \dfrac{1}{2} \int_{\Omega}\int_{\Omega}J(x,y) (u(x)-u(y))(v(x)-v(y)) \chi_{\varepsilon}(|x-y|) dx dy,
	\end{align*}
	which implies
	\begin{align*}
		\int_{\Omega} v \L_{\varepsilon} u dx  &=  \dfrac{1}{2} \int_{\Omega}\int_{\Omega}J(x,y) (u(x)-u(y))(v(x)-v(y)) \chi_{\varepsilon}(|x-y|) dx dy \\
		&\quad + \int_{\Omega} B(x)u(x)v(x) dx.
	\end{align*}
	Notice that $|v \L_{\varepsilon} u| \le |v \varphi_u|$ and $v\varphi_u \in L^1(\Omega)$ by assumption \eqref{eq_L2bis}. Furthermore, we have
	\begin{align*}
		&|u(x)-u(y)||v(x)-v(y)| J(x,y)\chi_{\varepsilon}(|x-y|) \le |u(x)-u(y)||v(x)-v(y)| J(x,y),\\
		& \int_{\Omega}\int_{\Omega}|u(x)-u(y)||v(x)-v(y)| J(x,y) dxdy < +\infty,
	\end{align*}
and $\chi_{\varepsilon}(|x-y|) \to 1$ as $\varepsilon \to 0$. Hence, letting $\varepsilon \to 0^+$ and using the dominated convergence theorem, we conclude that
	\begin{align*}
	  \int_{\Omega} v \L u dx 
		&= \lim_{\varepsilon \to 0^+} \int_{\Omega} v \L_{\varepsilon} u dx  \\
		&= \lim_{\varepsilon \to 0^+}  \dfrac{1}{2} \int_{\Omega}\int_{\Omega} [u(x)-u(y)][v(x)-v(y)] J(x,y)\chi_{\varepsilon}(|x-y|) dx dy  + \int_{\Omega} B(x)u(x)v(x) dx \\
		&= \dfrac{1}{2} \int_{\Omega}\int_{\Omega} [u(x)-u(y)][v(x)-v(y)] J(x,y) dx dy + \int_{\Omega} B(x)u(x)v(x) dx,
	\end{align*}
	which is the desired result. 
\end{proof}

The following lemma lies at the core of the proof of Kato's inequality. 

\begin{lemma}[Kato's inequality in function framework] \label{lem:kato} Assume that \eqref{eq_L1}--\eqref{eq_L4} and \eqref{eq_G2} hold. Let $f \in L^1(\Omega,\delta^{\gamma})$, $u = \Gbb^{\Omega} [f]$ and $p \in C^{1,1}(\R)$ be a convex function such that $p(0) = p'(0) = 0$ and $|p'| \le 1$. Then 
\begin{equation}\label{eq:kato1}
\int_{\Omega} p(u) \xi dx \le \int_{\Omega} fp'(u) \Gbb^{\Omega}[\xi] dx,
\end{equation}
for every $\xi \in \delta^{\gamma}L^{\infty}(\Omega)$ such that $\Gbb^{\Omega}[\xi] \ge 0$ a.e. in $\Omega$.
\end{lemma}

\begin{proof}  The proof is divided into two steps.  Recall that by \cite[Proposition 5.2]{TruongTai_2020}, under assumptions \eqref{eq_L1}--\eqref{eq_L2} and \eqref{eq_G2}, for $w \in L^2(\Omega)$  one has $\Gbb^{\Omega}[w] \in \Hbb(\Omega)$ and 
\begin{equation}\label{eq:variational}
	\int_{\Omega} w\zeta dx = \inner{\Gbb^{\Omega}[w],\zeta}_{\Hbb(\Omega)},\quad \forall \zeta \in \Hbb(\Omega).
\end{equation}

\textbf{Step 1.} Consider the case $f \in C^{\infty}_c(\Omega)$.  In this case, $u = \Gbb^{\Omega}[f] \in \Hbb(\Omega)$ and since $p'$ is Lipschitz with $|p'| \le 1$, we have $p'(u) \in \Hbb(\Omega) \cap L^{\infty}(\Omega)$ by Lemma \ref{lem:Hs00}.  On the other hand, $\Gbb^{\Omega}[\xi] \in \Hbb(\Omega) \cap L^{\infty}(\Omega)$ for every $\xi \in \delta^{\gamma}L^{\infty}(\Omega)$ (see \cite[Proposition 4.11 and Proposition 5.2]{TruongTai_2020}). This together with Lemma \ref{lem:Hs00} again implies $p'(u) \Gbb^{\Omega}[\xi] \in \Hbb(\Omega)$.

In \eqref{eq:variational}, choosing $\zeta=p'(u) \Gbb^{\Omega}[\xi] \in \Hbb(\Omega)$ and $w= f \in C^{\infty}_c(\Omega) \subset L^2(\Omega)$, we obtain
\begin{equation}\label{eq:equa1}
\begin{aligned}
\int_{\Omega} fp'(u) \Gbb^{\Omega}[\xi] dx &= \inner{ u, p'(u)\Gbb^{\Omega}[\xi] }_{\Hbb(\Omega)} 
\\ &= {\frac{1}{2}}\int_{\Omega} \int_{\Omega}  V(x,y)J(x,y)dxdy + \int_{\Omega} B u p'(u) \Gbb^{\Omega}[\xi] dx,
\end{aligned}
\end{equation}
where 
$$V(x,y):= [ u(x) - u(y)]\left[p'(u(x))\Gbb^{\Omega}[\xi](x) -  p'(u(y))\Gbb^{\Omega}[\xi](y)\right], \quad x,y \in \Omega.$$
The second equality in \eqref{eq:equa1} follows from \eqref{bilinear} and \eqref{innerH}.

Put
\begin{align*}
	&a_1: = \Gbb^{\Omega}[\xi](x) ,  \quad b_1: = (u(x)- u(y))p'(u(x)) - (p(u(x)) - p(u(y))),\\
	&a_2: = \Gbb^{\Omega}[\xi](y),  \quad b_2: = (u(x)-u(y))p'(u(y)) - (p(u(x)) - p(u(y))).
\end{align*}
Then we can write $V$ as
$$
V(x,y) = a_1b_1 - a_2b_2 +\left(p(u(x)) - p(u(y))\right)\left(\Gbb^{\Omega}[\xi](x) - \Gbb^{\Omega}[\xi](y)\right), \quad x,y \in \Omega.
$$
Since $p$ is convex, for every $x,y \in \Omega$, we have
\begin{align*}
p(u(x)) - p(u(y)) &\le  p'(u(x)) (u(x) - u(y)), \\ p(u(x)) - p (u(y)) &\ge p'(u(y))(u(x) - u(y)),
\end{align*}
which implies $b_1 \ge 0 \ge b_2$.  Noticing that $a_1, a_2 \ge 0$, we have $a_1b_1 \ge a_2b_2$, which yields
\begin{equation}\label{eq:compareO1}
\begin{aligned}
V(x,y) \ge (p(u(x)) - p(u(y)))(\Gbb^{\Omega}[\xi](x) - \Gbb^{\Omega}[\xi](y)),\quad x,y \in \Omega.
\end{aligned}
\end{equation}
Since $p$ is convex and $p(0) = 0$, it follows that $p(t) \le t p'(t), t \in \R$. Combining this with \eqref{eq:equa1} and \eqref{eq:compareO1},  we deduce
\begin{equation}\label{eq:equa3}
\begin{aligned}
\int_{\Omega} fp'(u) \Gbb^{\Omega}[\xi] dx &\geq {\frac{1}{2}}\int_{\Omega} \int_{\Omega} [p(u(x))-p(u(y))][\Gbb^{\Omega}[\xi](x) - \Gbb^{\Omega}[\xi](y)]J(x,y)dx dy \\ &\quad+ \int_{\Omega} Bp(u)\Gbb^{\Omega}[\xi] dx \\  &= \inner{ p(u), \Gbb^{\Omega} [\xi]}_{\Hbb(\Omega)}.
\end{aligned}
\end{equation}
Here we have used the fact that $p(u) \in  \Hbb(\Omega)$ by Lemma \ref{lem:Hs00} with $p(0) = 0$, $p$ is Lipschitz. Finally, replacing $w$ by $\xi \in \delta^{\gamma}L^{\infty}(\Omega) \subset L^2(\Omega)$ and $\zeta$ by $p(u)\in \Hbb(\Omega)$ in \eqref{eq:variational}, we deduce that
\begin{equation}\label{eq:kato1bis}
\int_{\Omega} fp'(u) \Gbb^{\Omega}[\xi] dx \ge \inner{ p(u), \Gbb^{\Omega} [\xi]}_{\Hbb(\Omega)} = \int_{\Omega} p(u) \xi dx,\quad \forall \xi \in \delta^{\gamma}L^{\infty}(\Omega), \Gbb^{\Omega}[\xi] \ge 0.
\end{equation}
Hence, \eqref{eq:kato1} follows for $f \in C^{\infty}_c(\Omega)$.

\textbf{Step 2.} Consider $f \in L^1(\Omega,\delta^{\gamma})$. Let  $\{ f_n\}_{n \in \mathbb{N}} \subset C^{\infty}_c(\Omega)$ be a sequence converging to $f$ in $L^1(\Omega,\delta^{\gamma})$ and a.e. in $\Omega$. Set $u=\Gbb^{\Omega}[f]$ and $u_n = \Gbb^{\Omega}[f_n]$, $n \in \mathbb{N}$. Since $f_n \to f$ in $L^1(\Omega,\delta^{\gamma})$ and the map $\Gbb^{\Omega}: L^1(\Omega,\delta^{\gamma}) \to L^1(\Omega,\delta^{\gamma})$ is continuous, up to a subsequence, we deduce that $u_n \to u \text{ in }L^1(\Omega,\delta^{\gamma}) \text{ and a.e. in }\Omega$. 

By the assumption $p \in C^{1,1}(\R)$ and $|p'| \leq 1$, we obtain $|p(u_n)-p(u)| \leq |u_n-u|$ in $\Omega$, which implies that $p(u_n) \to p(u)$ in $L^1(\Omega,\delta^\gamma)$ and a.e. in $\Omega$. In particular,
\begin{equation} \label{puun}
\lim_{n \to \infty}\int_{\Omega}p(u_n)\xi dx = \int_{\Omega}p(u)\xi dx, \quad \forall \xi \in \delta^\gamma L^\infty(\Omega).	
\end{equation}

Now, notice that $f_np'(u_n) \to fp'(u) \text{ a.e. in }\Omega$, $|f_n p'(u_n)| \le |f_n|$ and $f_n \to f$ in $L^1(\Omega,\delta^{\gamma})$. By using the generalized dominated convergence theorem, we obtain $f_np'(u_n) \to fp'(u)$ in $L^1(\Omega,\delta^\gamma)$. For any $\xi \in \delta^\gamma L^\infty(\Omega)$, we  have $\Gbb^{\Omega}[\xi] \in \delta^\gamma L^\infty(\Omega)$ due to \cite[Proposition 3.5]{ChaGomVaz_2019_2020}. Therefore
\begin{equation} \label{fpuun}
\lim_{n \to \infty}\int_{\Omega}f_np'(u_n) \Gbb^{\Omega}[\xi] dx = \int_{\Omega}fp'(u) \Gbb[\xi] dx, \quad \forall \xi \in \delta^\gamma L^\infty(\Omega).	
\end{equation}
Applying \eqref{eq:kato1bis} with $f$ replaced by $f_n \in C^{\infty}_c(\Omega)$, we obtain
\begin{equation}\label{eq:kato2}
\int_{\Omega} f_n p'(u_n) \Gbb^{\Omega}[\xi] dx \ge \int_{\Omega} p(u_n) \xi dx,\quad \forall \xi \in \delta^{\gamma}L^{\infty}(\Omega), \Gbb^{\Omega} [\xi]\ge 0.
\end{equation}
Combining \eqref{puun}--\eqref{eq:kato2}, we derive
\begin{align*}
\int_{\Omega} p(u) \xi dx = \lim_{n \to \infty}  \int_{\Omega} p(u_n) \xi dx  &\leq  \lim_{n \to \infty} \int_{\Omega} f_n p'(u_n) \Gbb^{\Omega} [\xi] dx \\
&= \int_{\Omega} f p'(u) \Gbb^{\Omega}[\xi] dx,\quad \forall \xi \in \delta^{\gamma}L^{\infty}(\Omega), \Gbb^{\Omega} [\xi]\ge 0.
\end{align*}
We complete the proof.
\end{proof}

\begin{theorem}[Kato's inequality in measure framework] \label{thm:kato} Assume that \eqref{eq_L1}--\eqref{eq_L4} and \eqref{eq_G2} hold. 
	
	$(i)$ Let $f \in L^1(\Omega,\delta^{\gamma})$ and put $u = \Gbb^{\Omega} [f]$. Then
	\begin{equation}\label{eq:kato_abs}
		\int_{\Omega} |u|\xi dx \le \int_{\Omega} \sgn(u) f \Gbb^{\Omega}[\xi]dx,
	\end{equation}
	\begin{equation}\label{eq:kato_main}
		\int_{\Omega} u^+\xi dx \le \int_{\Omega} \sgn^+(u) f \Gbb^{\Omega}[\xi]dx,
	\end{equation}
	for all $\xi \in \delta^{\gamma}L^{\infty}(\Omega)$ such that $\Gbb^{\Omega}[\xi] \ge 0$ a.e. in $\Omega$.
	
	$(ii)$ Assume in addition that \eqref{eq_G3bis} and \eqref{eq_G4} hold. Let $f \in L^1(\Omega,\delta^{\gamma})$, $\mu \in \M(\Omega,\delta^\gamma)$ and put $u = \Gbb^{\Omega} [f] + \Gbb^{\Omega}[\mu]$. Then
	\begin{equation}\label{eq:kato_main2}
		\int_{\Omega} u^+\xi dx \le \int_{\Omega} \sgn^+(u) \Gbb^{\Omega}[\xi]fdx + \int_{\Omega} \Gbb^{\Omega}[\xi] d\mu^+ ,
	\end{equation}
	\begin{equation}\label{eq:kato_abs2}
		\int_{\Omega} |u|\xi dx \le \int_{\Omega} \sgn(u) \Gbb^{\Omega}[\xi]fdx + \int_{\Omega} \Gbb^{\Omega}[\xi] d\mu^+ ,
	\end{equation}
	for all $\xi \in \delta^{\gamma}L^{\infty}(\Omega)$ such that $\Gbb^{\Omega}[\xi] \ge 0$ a.e. in $\Omega$.
\end{theorem}

\begin{remark} Various versions of Kato's inequality have been established for the classical Laplacian \cite{MarVer_2014}, for the RFL \cite[Proposition 2.4]{CheVer_2014} and for the SFL \cite[Lemma 31]{AbaDup_2017} where weighted, continuous up to the boundary functions are chosen as test functions. In contrast, our proof of Theorem \ref{thm:kato} relies on formula \eqref{LJB} and the properties of the space $\Hbb(\Omega)$, which allows us to deal with different types of nonlocal operators such as RFL, SFL and CFL. When measures are involved, \eqref{eq:kato_main2} and \eqref{eq:kato_abs2} are proved by an approximation argument in which the stability result Corollary \ref{cor:stability} is employed under additional conditions \eqref{eq_G3bis} and \eqref{eq_G4}.

As far as we know, in the literature, different variants of Kato's inequality are proved in the context of certain notion of solutions to linear equations. However, Theorem \ref{thm:kato} is stated for functions admitting a Green representation without appealing prematurely any notion of solutions. Thereby, unnecessary potential restriction on the application of Theorem \ref{thm:kato} might be avoided. For instance, Kato's inequality \eqref{eq:kato_main} can be applied to show the uniqueness of weak solutions of semilinear equations involving RFL in \cite{CheVer_2014} due to the fact that these solutions admit a Green representation. 
\end{remark}

\begin{proof}[{\sc Proof of Proposition \ref{thm:kato}}{\normalfont (i)}] Consider the sequence $\{ p_k \}_{k \in \mathbb{N}}$ given by
\begin{equation}\label{eq:sequencepk}
p_k(t) := \left\{ \begin{aligned} &|t| - \frac{1}{2k} &&\text{ if }|t| \ge \frac{1}{k},\\[3pt] 
&\frac{kt^2}{2} &&\text{ if } |t| < \frac{1}{k}. 
\end{aligned} \right.
\end{equation}
Then for every $k \in \mathbb{N}$, $p_k \in C^{1,1}(\R)$ is convex, $p_k(0) = (p_k)'(0) = 0$ and $|(p_k)'| \le 1$.  Hence, applying Lemma \ref{lem:kato} with $p = p_k$, one has
\begin{equation} \label{pku}
\int_{\Omega} p_k(u) \xi dx \le \int_{\Omega} f (p_k)'(u) \Gbb^{\Omega}[\xi] dx \le \int_{\Omega} |f| \Gbb^{\Omega}[\xi] dx,\quad \forall \xi \in \delta^{\gamma}L^{\infty}(\Omega), \Gbb^{\Omega} [\xi]\ge 0.
\end{equation}
Notice that $p_k(t) \to |t|$ and $(p_k)'(t) \to \sgn (t)$ as $k \to \infty$.  Hence, letting $k \to \infty$ in \eqref{pku} and using the dominated convergence theorem, we obtain \eqref{eq:kato_abs}.  Finally, by the integration-by-parts formula (see \cite[Lemma 4.4]{TruongTai_2020}), we have 
\[\int_{\Omega} u \xi dx = \int_{\Omega} f \Gbb^{\Omega}[\xi]dx, \quad \forall \xi \in \delta^{\gamma}L^{\infty}(\Omega).\]
This and \eqref{eq:kato_abs} imply \eqref{eq:kato_main}. The proof is complete.
\end{proof}

We next prove a version of Kato's inequality involving measures.

\begin{proof}[{\sc Proof of Proposition~\ref{thm:kato}}{\normalfont (ii)}] Let $\{ \mu_{i,n}\}_{n \in \N} \subset L^1(\Omega,\delta^{\gamma})$, $i=1,2$, be sequences of nonnegative functions such that $\{\mu_{1,n}\}_{n \in \mathbb{N}}$ and $\{\mu_{2,n}\}_{n \in \mathbb{N}}$ converge weakly to $\mu^+$ and $\mu^-$ in $\M(\Omega,\delta^\gamma)$ respectively. The existence of $\{ \mu_{i,n}\}_{n \in \N}$ is standard by noticing that the measures $\delta^\gamma \mu^+ \in \M(\Omega)$ and $\delta^\gamma \mu^- \in \M(\Omega)$ can be approximated by sequences $\{\psi_{1,n}\}_{n \in \N} \subset C_c^\infty(\Omega)$ and $\{\psi_{2,n}\}_{n \in \N} \subset C_c^\infty(\Omega)$ respectively in the weak* sense and by putting $\mu_{i,n}=\psi_{i,n}\delta^{-\gamma}$.
	
	  Put $\mu_n:=\mu_{1,n}- \mu_{2,n}$ then $\{\mu_n\}_{n \in \mathbb{N}} \subset L^1(\Omega,\delta^\gamma)$ and  $\{\mu_n\}_{n \in \mathbb{N}}$ converges weakly to $\mu$ in $\M(\Omega,\delta^\gamma)$. Denote $u_n := \Gbb^{\Omega}[f + \mu_n]$, $n \in \mathbb{N}$. By Corollary \ref{cor:stability}, up to a subsequence we have $u_n \to u$ in $L^1(\Omega,\delta^{\gamma})$ and a.e. in $\Omega$. Using Lemma \ref{lem:kato} and the estimate $|p'| \leq 1$, we obtain
\begin{align*}
\int_{\Omega} p(u_n)\xi dx &\le \int_{\Omega} p'(u_n) \Gbb^{\Omega}[\xi]fdx + \int_{\Omega} p'(u_n)\Gbb^{\Omega}[\xi] \mu_n dx\\
&\le \int_{\Omega} p'(u_n) \Gbb^{\Omega}[\xi]fdx + \int_{\Omega} \Gbb^{\Omega}[\xi]  \mu_{1,n} dx,
\end{align*}
for every $\xi \in \delta^{\gamma}L^{\infty}(\Omega)$ such that $\Gbb^{\Omega} [\xi] \ge 0$ a.e. in $\Omega$.  Noticing that $p \in C^{1,1}(\R)$, $u_n \to u \text{ in }L^1(\Omega,\delta^{\gamma})$ and a.e. in $\Omega$, $p(u_n) \to p(u)$ in $L^1(\Omega,\delta^\gamma)$ and $\{\mu_{1,n}\}_{n \in \mathbb{N}}$ converges weakly to $\mu^+$,
we can pass to the limit to get
\begin{equation}\label{eq:corolkato1}
\int_{\Omega} p(u)\xi dx \le \int_{\Omega} p'(u) \Gbb^{\Omega}[\xi]fdx + \int_{\Omega} \Gbb^{\Omega}[\xi] d\mu^+,\quad \forall \xi \in \delta^{\gamma}L^{\infty}(\Omega), \Gbb^{\Omega}[\xi] \ge 0.
\end{equation}
Now consider the sequence $\{ p_k\}_{k \in \mathbb{N}}$ as in \eqref{eq:sequencepk}.  In \eqref{eq:corolkato1},  replacing $p$ by $p_k$ and letting $k \to \infty$,  we derive \eqref{eq:kato_abs2}. 

Next, by the integration-by-parts formula (see \cite[Lemma 4.4]{TruongTai_2020}), we have 
$$\int_{\Omega} u \xi dx = \int_{\Omega} f \Gbb^{\Omega}[\xi]dx + \int_{\Omega}  \Gbb^{\Omega}[\xi]d\mu, \quad \forall \xi \in \delta^{\gamma}L^{\infty}(\Omega).$$
This and \eqref{eq:kato_abs2} imply
\eqref{eq:kato_main2}. The proof is complete.
\end{proof}
\section{Semilinear elliptic equations with measure data}\label{sec:semilinear}
In this section,  we study the existence and uniqueness of solutions to problem \eqref{eq:semilinear_absorption}.  Recall that a function $u$ is a weak-dual solution of \eqref{eq:semilinear_absorption} if $u \in L^1(\Omega,\delta^{\gamma})$, $g(u) \in  L^1(\Omega,\delta^{\gamma})$ and
$$\int_{\Omega} u\xi dx + \int_{\Omega} g(u)\Gbb^{\Omega}[\xi] dx = \int_{\Omega} \Gbb^{\Omega}[\xi] d\mu,\quad \forall \xi \in \delta^{\gamma}L^{\infty}(\Omega).$$

The idea of the proof of Theorem \ref{theo:subsupersolution} is based on method of sub- and supersolutions introduced in \cite{MonPon_2008}.  We begin with the following equi-integrability result.
\begin{lemma}
Assume $v_1,v_2 \in L^1(\Omega,\delta^\gamma)$ such that $v_1 \leq v_2$ and $g(v_1), g(v_2) \in L^1(\Omega,\delta^\gamma)$. 
Then the set 
$$\mathcal{F}:= \left\{ g(v) \in L^1(\Omega,\delta^{\gamma}) :v \in L^1(\Omega,\delta^{\gamma}) \text{ and }v_1 \le v \le v_2 \text{ a.e. in }\Omega\right\}$$
is equi-integrable. 
\end{lemma}

\begin{proof} The proof follows that of \cite[Proposition 2.1]{MonPon_2008}. Here we provide it for the sake of convenience. Assume that $\mathcal{F}$ is not equi-integrable in $L^1(\Omega,\delta^{\gamma})$. Then there exist $\varepsilon > 0$, a sequence $\{ u_n\}_{n \in \mathbb{N}}$ such that $v_1 \le u_n \le v_2$ a.e. in $\Omega$, and a sequence of measurable subsets $\{ E_n\}_{n \in \mathbb{N}}$ of $\Omega$ such that
$$|E_n| \to 0 \text{ as } n \to \infty \quad \text{ and } \quad \int_{E_n} g(u_n) \delta^{\gamma} dx \ge \varepsilon, \quad \forall n \in \mathbb{N}.$$
By \cite[Lemma 2.1]{MonPon_2008} with $w_n:= g( u_n)\delta^{\gamma}/\varepsilon$,  we can choose a subsequence $\{u_{n_k}\}_{k \in \mathbb{N}}$ and a sequence of disjoint measurable sets $\{F_k\}_{k \in \mathbb{N}}$ such that
$$\int_{F_k} |g(u_{n_k})|\delta^{\gamma} dx \ge \dfrac{\varepsilon}{2},\quad \forall k \in \mathbb{N}.$$
Put
$$
v(x):= \va{ &u_{n_k}(x) &\text{ if } x\in F_k \text{ for some }k \ge 1,\\
&v_1(x) &\text{ otherwise.}}
$$
Then $v \in L^1(\Omega,\delta^{\gamma})$ and $v_1 \le v \le v_2$. In addition,
$$\int_{\Omega} |g(v)|\delta^{\gamma} dx \ge \sum_{k=1}^{\infty} \int_{\Omega} |g(u_{n_k})|\delta^{\gamma} dx = +\infty,$$
which is a contradiction.
\end{proof}

As a result,  by following the idea in \cite[Theorem 2.1]{MonPon_2008}, we obtain
\begin{proposition}\label{prop:continuousCara} Assume that $g(v) \in L^1(\Omega,\delta^{\gamma})$  for every $v\in L^1(\Omega,\delta^{\gamma})$.
Let
$\mathcal{A}$ be the operator defined by $\mathcal{A}v(x) = g(v(x))$ for $v \in L^1(\Omega,\delta^{\gamma})$ and $x \in \Omega$. Then $\mathcal{A}: L^1(\Omega,\delta^{\gamma}) \to L^1(\Omega,\delta^{\gamma})$  is continuous.
\end{proposition}

\begin{proof}[\sc Proof of Theorem \ref{theo:subsupersolution}] For any measurable function $v$, set
$$h(v(x)) := \va{ 
&g(\Gbb^{\Omega} [\mu^+](x)) &\text{ if }  v(x) > \Gbb^{\Omega}[\mu^+](x),\\
&g(v(x)) &\text{ if } -\Gbb^{\Omega}[\mu^-](x) \le v(x) \le \Gbb^{\Omega}[\mu^+](x),\\
&g(-\Gbb^{\Omega} [\mu^-](x)) &\text{ if } v(x) < -\Gbb^{\Omega} [\mu^-](x).}$$
By condition \eqref{eq:goodmeasure}, we have $h(v) \in L^1(\Omega,\delta^\gamma)$. Hence, the map $\mathcal{A}: L^1(\Omega,\delta^{\gamma}) \to L^1(\Omega,\delta^{\gamma})$  defined by $\mathcal{A}v(x) = h(v(x))$ is continuous by Proposition \ref{prop:continuousCara}. Furthermore,  $g(-\Gbb^{\Omega}[\mu^-]) \le h(v) \le g(\Gbb^{\Omega}[\mu^+]),\forall v \in L^1(\Omega,\delta^{\gamma})$.

\textbf{Step 1.} Assume that $\mu \in \M(\Omega,\delta^{\gamma})$. We prove that there exists a function $u \in L^1(\Omega,\delta^{\gamma})$ such that
$u + \Gbb^{\Omega}[h(u)] = \Gbb^{\Omega}[\mu]$.
Indeed, consider the operator $\mathbb{T} : L^1(\Omega,\delta^{\gamma}) \to L^1(\Omega,\delta^{\gamma})$ defined by 
$$\mathbb{T} u := \Gbb^{\Omega}\left[\mu - h( u)\right], \quad u \in L^1(\Omega,\delta^{\gamma})
$$ 
and the set 
$$\mathcal{C}:=\left\{ u \in L^1(\Omega,\delta^{\gamma}): \norm{u}_{L^1(\Omega,\delta^{\gamma})} \le M  \right\},$$ 
where 
$$M:= C_1\left(\norm{\mu}_{\M(\Omega,\delta^{\gamma})} +  \norm{g\left(\Gbb^{\Omega}[\mu^+]\right)}_{L^1(\Omega,\delta^{\gamma})} +  \norm{g\left(-\Gbb^{\Omega}[\mu^-]\right)}_{L^1(\Omega,\delta^{\gamma})} \right),$$ 
with $C_1=C_1(\Omega,N,s)$ being the constant in the estimate $\norm{\Gbb^{\Omega}[\mu]}_{L^1(\Omega,\delta^{\gamma})} \le C_1 \norm{\mu}_{\M(\Omega,\delta^{\gamma})}$.

Since $\Gbb^{\Omega}: \M(\Omega,\delta^{\gamma}) \to L^1(\Omega,\delta^{\gamma})$ is compact by Theorem \ref{theo:compactness} and the map $\mathcal{A}: L^1(\Omega,\delta^{\gamma}) \to L^1(\Omega,\delta^{\gamma})$ is continuous, it follows that $\mathbb{T}$ is compact. Furthermore,  $\mathcal{C}$ is closed, bounded, convex and $\mathbb{T}(\mathcal{C}) \subset \mathcal{C}$ since
\begin{align*}
\norm{\T u}_{L^1(\Omega,\delta^{\gamma})} 
= \norm{\Gbb^{\Omega}\left[\mu - h( u)\right]}_{L^1(\Omega,\delta^{\gamma})} 
\leq  C_1\left(\norm{\mu}_{\M(\Omega,\delta^{\gamma})} + \norm{h(u)}_{L^1(\Omega,\delta^{\gamma})}\right) \le  M,
\end{align*}
for every $u \in \mathcal{C}$. By Schauder's fixed point theorem, we conclude that there exists a function $u \in L^1(\Omega,\delta^{\gamma})$ such that
$u + \Gbb^{\Omega}[h(u)] = \Gbb^{\Omega}[\mu]$.

\textbf{Step 2.} We prove that \begin{equation} \label{Gmupm}-\Gbb^{\Omega}[\mu^-] \le u \le \Gbb^{\Omega}[\mu^+] \quad \text{a.e. in } \Omega.
\end{equation}	
Denote $v := u - \Gbb^{\Omega}[\mu^+]$. We have
$$v = \Gbb^{\Omega}[-h(u)] + \Gbb^{\Omega}[\mu - \mu^+].$$
By Theorem \ref{thm:kato} (ii), we have
$$\int_{\Omega} v^+ \xi dx \leq -\int_{\left\{u \geq \Gbb^{\Omega}[\mu^+]\right\}} h(u) \Gbb^{\Omega}[\xi]dx = -\int_{\left\{u \geq \Gbb^{\Omega}[\mu^+]\right\}}  g(\Gbb^{\Omega}[\mu^+]) \Gbb^{\Omega}[\xi]dx \leq 0,$$
for every $\xi \in \delta^{\gamma}L^{\infty}(\Omega)$ such that $\Gbb^{\Omega}[\xi] \ge 0$, which implies $v^+ = 0$. Hence, $u \le \Gbb^{\Omega}[\mu^+]$. Similarly, one has $u \ge -\Gbb^{\Omega}[\mu^-]$. Thus, $h(u) = g(u)$ and $u$ satisfies $u  + \Gbb^{\Omega}[g(u)] = \Gbb^{\Omega}[\mu]$,
i.e. $u$ is a solution of \eqref{eq:semilinear_absorption} and \eqref{Gmupm} holds.

Next, we prove estimate \eqref{est:apriori}. Employing Kato's inequality \eqref{eq:kato_abs2}, we obtain
$$ \int_{\Omega} |u|\xi dx + \int_{\Omega} \sgn(u) g(u) \Gbb^{\Omega}[\xi] dx \leq  \int_{\Omega} \Gbb^{\Omega}[\xi] d\mu^+, \quad \forall \xi \in \delta^\gamma L^\infty(\Omega).
$$
Taking $\xi =\delta^\gamma$ and noting that $\Gbb^\Omega[\delta^\gamma] \asymp \delta^\gamma$ (see \cite[Example 3.6]{Aba_2019}) and $\sgn(u)g(u) = |g(u)|$,we deduce \eqref{est:apriori}.

\textbf{Step 3.} We prove that the solution is unique and the map $\mu \mapsto u$ is increasing. Assume that $\mu_1,\mu_2 \in \M(\Omega,\delta^{\gamma})$ such that $\mu_1 \le \mu_2$ and let $u_1$ and $u_2$ be solutions to \eqref{eq:semilinear_absorption} with measure $\mu_1$ and $\mu_2$ respectively. Then we have
$$u_1 - u_2 + \Gbb^{\Omega}[g(u_1) - g(u_2)] = \Gbb^{\Omega}[\mu_1 - \mu_2].$$
Applying Kato's inequality \eqref{eq:kato_main2} for $f = g(u_2) - g(u_1)$ and $\mu = \mu_1 - \mu_2$, we have
$$\int_{\Omega}(u_1-u_2)^+ \xi dx + \int_{ \{ u_1 \ge u_2 \} } \left( g(u_1) - g(u_2)\right) \Gbb^{\Omega}[\xi] dx \le \int_{\Omega} \Gbb^{\Omega}[\xi] d(\mu_1 - \mu_2)^+ = 0$$
for any $\xi \in \delta^{\gamma}L^{\infty}(\Omega)$ such that $\Gbb^{\Omega}[\xi] \ge 0.$
Since $g$ is increasing, we derive
$\int_{\Omega}(u_1 - u_2)^+ \xi dx \le 0$
for every $\xi \in \delta^{\gamma}L^{\infty}(\Omega)$ such that $\Gbb^{\Omega}[\xi]\ge 0$. This implies that $u_1 \le u_2$. Thus the map $\mu \to u$ is increasing. The uniqueness for \eqref{eq:semilinear_absorption} follows straightforward. We complete the proof.
\end{proof}

As a consequence of the above result, we obtain the unique existence for $L^1$ data.

\begin{proof}[\sc Proof of Corollary \ref{cor:L1data}] First we assume that $f \in L^{\infty}(\Omega)$.  We have $\Gbb^{\Omega}[|f|] \in L^{\infty}(\Omega)$ (see \cite[Theorem 2.1]{Aba_2019}),  which implies $g(\Gbb^{\Omega}[f^+]) \in L^{\infty}(\Omega)$ and $g(-\Gbb^{\Omega}[f^-]) \in L^{\infty}(\Omega)$, in particular $g(-\Gbb^{\Omega}[f^-])$ and $g(\Gbb^{\Omega}[f^+])$ belong to $L^1(\Omega,\delta^{\gamma})$. By Theorem \ref{theo:subsupersolution},  there exists a unique weak-dual solution to \eqref{eq:semilinear_absorption}.

Now consider the case $f \in L^1(\Omega,\delta^{\gamma})$ and let $\{f_n\}_{n \in \mathbb{N}}$ be a sequence in $L^{\infty}(\Omega)$ that converges to $f$ in $L^1(\Omega,\delta^{\gamma})$. Denote by $u_n$ the unique solution of \eqref{eq:semilinear_absorption} with datum $f_n$.  Since
$$u_n - u_m + \Gbb^{\Omega}[g(u_n) - g(u_m)] = \Gbb^{\Omega}[f_n - f_m],\quad m,n \in \mathbb{N},$$
using Kato's inequality \eqref{eq:kato_abs} and the assumption that $g$ is nondecreasing, we obtain 
$$\int_{\Omega} |u_n - u_m| \xi dx + \int_{\Omega}|g(u_n) - g(u_m)|\Gbb^{\Omega}[\xi]dx \le \int_{\Omega}|f_n - f_m|\Gbb^{\Omega}[\xi] dx,$$ for every $\xi \in \delta^{\gamma} L^{\infty}(\Omega), \Gbb^{\Omega}[\xi]\ge 0$. Choosing $\xi = \delta^{\gamma}$ and noticing that $\Gbb^{\Omega}[\delta^{\gamma}] \asymp \delta^{\gamma}$ by \cite[example 3.6]{Aba_2019}, one has
\begin{align*}
\int_{\Omega} |u_n - u_m| \delta^{\gamma} dx + \int_{\Omega}|g(u_n) - g(u_m)| \delta^{\gamma} dx  \lesssim \int_{\Omega} |f_n - f_m|\delta^{\gamma} dx.
\end{align*}
Since $\{f_n\}_{n \in \mathbb{N}}$ is a Cauchy sequence, we infer from the above estimate that $\{u_n\}_{ n \in \mathbb{N}}$ and $\{ g(u_n) \}_{n \in \mathbb{N}}$ are also  Cauchy sequences. This implies that, up to a subsequence,
$u_n \to u$ and $g(u_n) \to g(u)$ in $L^1(\Omega,\delta^{\gamma})$. 
Hence, letting $n \to \infty$ in the formula $u_n + \Gbb^\Omega[g(u_n)]=\Gbb^\Omega[f_n]$, we conclude that 
$u + \Gbb^{\Omega}[g(u)] = \Gbb^{\Omega}[f]$, 
which means that $u$ is a solution to \eqref{eq:semilinear_absorption} with $\mu=f$.  Finally, the solution is unique by Kato's inequality. We complete the proof.
\end{proof}

Next we give a sharp condition on $g$ under which \eqref{eq:goodmeasure} is satisfied. Recall that $p^*=\frac{N+\gamma}{N+\gamma-2s}$.

\begin{corollary} \label{cor:measuredata-sub} Assume that \eqref{eq_L1}--\eqref{eq_L4} and \eqref{eq_G2}--\eqref{eq_G4} hold. Let $g: \R \to \R$ be a nondecreasing continuous function satisfying \eqref{sub-int} and $g(0)=0$. Then for any $\mu \in \M(\Omega,\delta^{\gamma})$, problem \eqref{eq:semilinear_absorption} admits a unique weak-dual solution $u$. Furthermore, the solution satisfies \eqref{Gmu+-} and the map $\mu \mapsto u$ is nondecreasing.
\end{corollary}

\begin{lemma} \label{lem:subint}
Let $g: \R \to \R$ be a nondecreasing continuous function with $g(0)=0$. Assume that $g$ satisfies the subcritical integrability condition \eqref{sub-int}. 
\begin{equation} \label{sub-int1} \int_1^{\infty} [g(t) - g(-t)]t^{-1-p^*}dt < +\infty.
\end{equation}
Then for any $\mu \in \M(\Omega,\delta^\gamma)$,
$$ g(\Gbb^\Omega[|\mu|]), g(-\Gbb^\Omega[|\mu|]) \in L^1(\Omega,\delta^\gamma).
$$ 
\end{lemma}
\begin{proof} Let $\mu \in \M(\Omega,\delta^\gamma)$. It is sufficient to show that if $\mu \geq 0$ then $g(\Gbb^\Omega[\mu]) \in L^1(\Omega,\delta^\gamma)$.  For $\lambda>0$, set
$$ A_\lambda:=\{ x \in \Omega: \Gbb^\Omega[\mu](x) > \lambda \}, \quad a(\lambda) := \int_{A_\lambda}\delta^\gamma dx.
$$	
We write 
\begin{equation} \label{gG1-1} \begin{aligned}
\int_{\Omega}g(\Gbb^\Omega[\mu])\delta^\gamma dx &= \int_{\Omega \setminus A_1}g(\Gbb^\Omega[\mu])\delta^\gamma dx  +  \int_{A_1}g(\Gbb^\Omega[\mu])\delta^\gamma dx \\
&\leq g(1)\int_{\Omega}\delta^\gamma dx  - \int_1^{\infty} g(t)da(t).
\end{aligned} \end{equation}
For $\lambda_n \ge 1$, by integration by parts and estimate (which is deduced from \eqref{eq:Marcin} and the definition of Marcinkiewicz spaces and inequality \eqref{Marcin-equi})
\begin{equation} \label{atn} a(\lambda_n) \leq \lambda_n^{-p^*}\vertiii{ \Gbb^\Omega[\mu]}_{M^{p^*}(\Omega,\delta^\gamma)} \leq \lambda_n^{-p^*}\norm{ \Gbb^\Omega[\mu]}_{M^{p^*}(\Omega,\delta^\gamma)}.
\end{equation}
We read off from \eqref{sub-int1} that there exists a sequence $\{\lambda_n\}_{n \in \mathbb{N}}$ such that
\begin{equation} \label{tn} \lim_{\lambda_n \nearrow +\infty}\lambda_n^{-p^*}g(\lambda_n) = 0,
\end{equation}
which together with \eqref{atn} implies
$\lim_{\lambda_n \nearrow +\infty}a(\lambda_n)g(\lambda_n) = 0$.
Hence, one can write
\begin{align*}
- \int_1^{\lambda_n} g(t)da(t) &= - g(\lambda_n)a(\lambda_n) + g(1)a(1) + \int_{1}^{\lambda_n} a(t)dg(t) \\
&\leq - g(\lambda_n)a(\lambda_n) + g(1)a(1) + \norm{ \Gbb^\Omega[\mu]}_{M^{p^*}(\Omega,\delta^\gamma)}\int_{1}^{\lambda_n} t^{-p^*}dg(t) \\
&\leq - g(\lambda_n)a(\lambda_n) + g(1)a(1) + \norm{ \Gbb^\Omega[\mu]}_{M^{p^*}(\Omega,\delta^\gamma)}\int_{1}^{\lambda_n} t^{-p^*}dg(t) \\
&= \left(\norm{ \Gbb^\Omega[\mu]}_{M^{p^*}(\Omega,\delta^\gamma)}\lambda_n^{-p^*}- a(\lambda_n)\right)g(\lambda_n) +  \left(a(1)-\norm{ \Gbb^\Omega[\mu]}_{M^{p^*}(\Omega,\delta^\gamma)}\right)g(1) \\
&\quad +  p^*\norm{ \Gbb^\Omega[\mu]}_{M^{p^*}(\Omega,\delta^\gamma)} \int_1^{\lambda_n}g(t)t^{-p^*-1}dt. 
\end{align*}
Letting $\lambda_n \to +\infty$ and using \eqref{tn} and \eqref{atn}, we arrive at
$$ -\int_1^\infty g(t)da(t) \leq p^*\norm{ \Gbb^\Omega[\mu]}_{M^{p^*}(\Omega,\delta^\gamma)} \int_1^{\infty}g(t)t^{-p^*-1}dt.
$$ 
Plugging this into \eqref{gG1-1} and using \eqref{sub-int1}, we obtain
$$
\int_{\Omega}g(\Gbb^\Omega[\mu])\delta^\gamma dx   \leq g(1)\int_{\Omega}\delta^\gamma dx + p^*\norm{ \Gbb^\Omega[\mu]}_{M^{p^*}(\Omega,\delta^\gamma)} \int_1^{\infty}g(t)t^{-p^*-1}dt < +\infty,
$$
which means $g(\Gbb^\Omega[\mu]) \in L^1(\Omega,\delta^\gamma)$. 

In order to prove that $g(-\Gbb^\Omega[\mu]) \in L^1(\Omega,\delta^\gamma)$, we put $\tilde g(t) = -g(-t)$ and use the fact that $\tilde g(\Gbb^\Omega[\mu]) \in L^1(\Omega,\delta^\gamma)$. We complete the proof.
\end{proof}

\begin{proof}[{\sc Proof of Theorem \ref{measuredata-sub}}.]
This theorem follows from Theorem \ref{theo:subsupersolution} and Lemma \ref{lem:subint}.
\end{proof}

\section{Necessary and sufficient conditions for existence}\label{sec:sufficientcondition}
In this subsection, we assume that \eqref{eq_L1}--\eqref{eq_L4},  \eqref{eq_G2}--\eqref{eq_G4} hold.
In the sequel, for any measure in $\M(\Omega)$, we use the same notation to denote its extension by zero outside of $\Omega$. 
Let $R=\textrm{diam}(\Omega)+1$.  For $\alpha \in (0,\frac{N}{2})$, we define the truncated Riesz potential of a measure $\nu \in \M_+(\Omega)$ by
$$ \Ibb_{\alpha}^R[\nu](x):=\int_0^R \frac{\nu(B_t(x))}{t^{N-\alpha}}\frac{dt}{t}.
$$
The next lemma provides upper bounds of the Green operator in terms of the Riesz potential.
\begin{lemma}
Let $\mu \in \M_+(\Omega,\delta^\gamma)$,  $d\nu = \delta^\gamma d\mu$ and $p \in (1,\infty)$ such that 
$s-\frac{\gamma}{2p'}>0$.

(i) There exists a positive constant $C=C(N,\Omega,s,\gamma,p)$ such that
\begin{equation} \label{est:GW-1}
	\Gbb^{\Omega}[\mu](x) \leq C\, \delta(x)^{-\frac{\gamma}{p}} \Ibb_{2s-\frac{\gamma}{p'}}^{2R}[\nu](x) \quad \text{for } x \in \Omega.
\end{equation}

(ii) There exists a positive constant $C=C(N,\Omega,s,\gamma,p)$ such that
\begin{equation} \label{est:GW-2}
	\| \Gbb^{\Omega}[\mu]\|_{L^p(\Omega,\delta^\gamma)} \leq C \| \Ibb_{2s-\frac{\gamma}{p'}}^{2R}[\nu]\|_{L^p(\Omega)}. 
\end{equation}
\end{lemma}
\begin{proof}
(i) From the two-sided estimate of $G^\Omega$ in \eqref{G-est}, we obtain
$$
G^{\Omega}(x,y) \lesssim \delta(x)^{-\frac{\gamma}{p}}\delta(y)^{\gamma} |x-y|^{-(N-2s+\frac{\gamma}{p'})}, \quad x,y \in \Omega, \, x \neq y.
$$	
By using condition $s-\frac{\gamma}{2p'}>0$, the definition of the {\color{darkblue}Riesz potential} and Fubini's theorem, we infer from the above estimate that 
\begin{align*}
\Gbb^{\Omega}[\mu](x) &\lesssim \delta(x)^{-\frac{\gamma}{p}} \int_{\Omega} \delta(y)^{\gamma}(N-2s+ \frac{\gamma}{p'}) \int_{|x-y|}^\infty t^{-(N-2s+\frac{\gamma}{p'}+1)}	dt d\mu(y) \\
&\lesssim \delta(x)^{-\frac{\gamma}{p}} \int_{\Omega}   \int_0^\infty \1_{B_t(x)}(y) t^{-(N-2s+ \frac{\gamma}{p'}+1)}	dt \delta(y)^\gamma d\mu(y) \\
&\lesssim \delta(x)^{-\frac{\gamma}{p}} \int_0^\infty t^{-(N-2s+ \frac{\gamma}{p'}+1)} \int_{\Omega} \1_{B_t(x)}(y)  d\nu(y)	dt \\
&\lesssim \delta(x)^{-\frac{\gamma}{p}} \int_0^R \frac{\nu(B_t(x))}{t^{N-2s+ \frac{\gamma}{p'}}}\frac{dt}{t} + \delta(x)^{-\frac{\gamma}{p}} \nu(\Omega) \int_R^\infty t^{-(N-2s+ \frac{\gamma}{p'}+1)}dt \\
&\lesssim \delta(x)^{-\frac{\gamma}{p}} \int_0^R \frac{\nu(B_t(x))}{t^{N-2s+ \frac{\gamma}{p'}}}\frac{dt}{t} + \delta(x)^{-\frac{\gamma}{p}} \int_R^{2R} \frac{\nu(B_t(x))}{t^{N-2s+ \frac{\gamma}{p'}}}\frac{dt}{t} \\
& \asymp \delta(x)^{-\frac{\gamma}{p}} \int_0^{2R} \frac{\nu(B_t(x))}{t^{N-2s+ \frac{\gamma}{p'}}}\frac{dt}{t}.
\end{align*}
Therefore, we obtain \eqref{est:GW-1}.

(ii) From \eqref{est:GW-1}, we derive
\begin{align*}
\int_{\Omega} \delta(x)^\gamma \Gbb^{\Omega}[\mu](x)^p dx \lesssim \int_{\Omega} (\Ibb_{2s-\frac{\gamma}{p'}}^{2R}[\nu](x))^p dx, 	
\end{align*}
which implies \eqref{est:GW-2}.
\end{proof}

Before proving Theorem \ref{thm:supercritical}, we recall the definition of 
Bessel potential spaces and Bessel capacities. For $\alpha\in\R$ we define the Bessel kernel of order $\alpha$ in $\R^d$ by $\CB_{\alpha}(\xi):=\CF^{-1}((1+|\cdot|^2)^{-\frac{\alpha}{2}})(\xi)$,  where $\CF$ is the Fourier transform in the space $\CS'(\R^N)$ of moderate distributions in $\R^N$. For $\alpha \in \R$ and $q>1$, the Bessel potential space of order $\alpha$ and power is defined by
$$ L^{\alpha,q}(\R^N):=\{f=\CB_{\alpha} \ast g:g\in L^{q}(\R^N)\},
$$
endowed with the norm
$$ \|f\|_{L^{\alpha,q}(\R^N)}:=\|g\|_{L^q(\R^N)}=\|\CB_{-\alpha}\ast f\|_{L^p(\R^N)}.
$$
For $\alpha>0$ and $q \in (1,\infty)$, we denote by $L^{-\alpha, q'}(\R^N)$ the dual space of $L^{\alpha,q}(\R^N)$.  Also, for $E \subset \R^N$, we denote
$$
\CS_E:=\{g \in L^q(\R^N): g \geq 0, \, \CB_{\alpha} \ast g \geq \1_E\}.
$$
The Bessel capacity of a compact set $E$ is defined by
$$ \mathrm{Cap}_{\alpha,q}(E):=\inf\{\|g\|^q_{L^{\alpha, q}}: g \in \CS_E \},
$$
while the Bessel capacity of open sets and arbitrary sets can be defined in the standard way. We note that if $\CS_E = \varnothing$ then we set $\mathrm{Cap}_{\alpha,q}(E)=\infty$.

For a Radon measure $\nu$ on $\R^N$, put
$$
\Bb_\alpha[\nu](x) := \int_{\R^N}\CB_{\alpha}(x-y)d\nu(y), \quad x \in \R^N.
$$
The Bessel capacities can be also defined in a dual way involving measures as
$$ \mathrm{Cap}_{\alpha,q}(E)=\sup\left\{ \left( \frac{\nu(E)}{\| \Bb_\alpha[\nu] \|_{L^{q'}(\R^N)}} \right)^q: \nu \in \M^+(E) \right \}.
$$
For further properties of the Bessel potential spaces and the Bessel capacities, we refer to \cite[subsections 1.2.3, 1.2.4, 2.2, 2.5, 2.6]{Ada_1996} and \cite[subsection 3.3.1]{Ver-handbook}.

\begin{proof}[\sc Proof of Theorem \ref{thm:supercritical}] ~
	
\textbf{Step 1.} Construction of approximate solutions.	
	
We first observe that since $\nu$ is absolutely continuous with respect to the Bessel capacity $\mathrm{Cap}_{2s-\frac{\gamma}{p'},p'}$, the measure $\nu^+$ and $\nu^-$ inherit the same property.  Therefore, by \cite[Proposition 3.17]{Ver-handbook}, there exist nondecreasing sequences $\{ \nu_{i,n} \}_{n \in \N} \subset L^{-(2s-\frac{\gamma}{p'}),p}(\R^N) \cap \M^+(\Omega)$, $i=1,2$, with compact supports in $\Omega$ such that $\{\nu_{1,n}\}_{n \in \N}$ and  $\{\nu_{2,n}\}_{n \in \N}$ converge to $\nu^+$ and $\nu^-$ {\color{xanhle} in the weak$^*$ sense}. Moreover, by \cite[Theorem 2.3]{BVHV},
\begin{equation} \label{est:WB}
\| \Ibb_{2s-\frac{\gamma}{p'}}^{2R}[\nu_{i,n}]\|_{L^p(\R^N)} \asymp \| \Bb_{2s-\frac{\gamma}{p'}}[\nu_{i,n}] \|_{L^p(\R^N)}.
\end{equation}
Put $\mu_{i,n}=\delta^{-\gamma}\nu_{i,n}$ then $\mu_{i,n} \in \M(\Omega,\delta^\gamma)$. Combining \eqref{est:GW-2} and \eqref{est:WB} yields
$$ \| \Gbb^{\Omega}[\mu_{i,n}]\|_{L^p(\Omega,\delta^\gamma)} \lesssim \|\Bb_{2s-\frac{\gamma}{p'}}[\nu_{i,n}] \|_{L^p(\R^N)}  < +\infty.
$$ 
By Theorem \ref{theo:subsupersolution}, there exist unique weak-dual solutions $u_{1,n}$, $u_{2,n}$ and $u_{n}$ to the following problems respectively
\begin{equation}\label{eq:mu1n_1}
	\left\{ \begin{aligned}
		\L u + |u|^{p-1}u &= \mu_{1,n} &&\text{ in }\Omega, \\
		u  &= 0 &&\text{ on }\partial \Omega \text{ or in } \R^N \backslash \Omega \text{ (if applicable)},
	\end{aligned} \right.
\end{equation}
\begin{equation}\label{eq:mu1n_2}
	\left\{ \begin{aligned}
		\L u + |u|^{p-1}u &= \mu_{2,n} &&\text{ in }\Omega, \\
		u  &= 0 &&\text{ on }\partial \Omega \text{ or in } \R^N \backslash \Omega \text{ (if applicable)},
	\end{aligned} \right.
\end{equation}
\begin{equation}\label{eq:mu1n}
	\left\{ \begin{aligned}
		\L u + |u|^{p-1}u &= \mu_{1,n} - \mu_{2,n} &&\text{ in }\Omega, \\
		u  &= 0 &&\text{ on }\partial \Omega \text{ or in } \R^N \backslash \Omega \text{ (if applicable)}.
	\end{aligned} \right.
\end{equation}
In other words, $u_{i,n}, u_n \in L^1(\Omega,\delta^\gamma)$ satisfy the following weak-dual formulations respectively
\begin{equation} \label{eq:u_in} \int_{\Omega} u_{i,n}\xi dx + \int_{\Omega} |u_{i,n}|^{p-1}u_{i,n}\Gbb^{\Omega}[\xi] dx = \int_{\Omega} \Gbb^{\Omega}[\xi] d\mu_{i,n},\quad \forall \xi \in \delta^{\gamma}L^{\infty}(\Omega).
\end{equation} 
\begin{equation} \label{eq:u_n} \int_{\Omega} u_n\xi dx + \int_{\Omega} |u_{n}|^{p-1}u_{n}\Gbb^{\Omega}[\xi] dx = \int_{\Omega} \Gbb^{\Omega}[\xi] d\mu_{n},\quad \forall \xi \in \delta^{\gamma}L^{\infty}(\Omega).
\end{equation} 

\textbf{Step 2.} Convergence procedure.

By the monotonicity and positivity of the sequences $\{\mu_{i,n}\}_{n \in \mathbb{N}}$ (by the properties of $\{\nu_{i,n}\}_{n \in \mathbb{N}}$),  we deduce from Theorem \ref{theo:subsupersolution} that $u_{i,n} \geq 0$ in $\Omega$, $\{u_{i,n}\}_{n \in \N}$, $i=1,2$, are increasing sequences and 
\begin{equation} \label{order}
- \Gbb^{\Omega}[\mu^-] \leq - \Gbb^\Omega[\mu_{2,n}] \leq -u_{2,n} \leq u_n \leq u_{1,n} \leq \Gbb^\Omega[\mu_{1,n}] \leq \Gbb^\Omega[\mu_{+}] \quad \text{in } \Omega.	
\end{equation}	
Moreover, by \eqref{est:apriori}, 
\begin{equation} \label{est:apriori-u1} \| u_{1,n} \|_{L^1(\Omega,\delta^\gamma)} +  \| u_{1,n}  \|_{L^p(\Omega,\delta^\gamma)} \leq C\| \mu_{1,n} \|_{\M(\Omega,\delta^\gamma)} \leq C \| \mu^+ \|_{\M(\Omega,\delta^\gamma)}, 
\end{equation}
\begin{equation} \label{est:apriori-u2} \| u_{2,n} \|_{L^1(\Omega,\delta^\gamma)} +  \|u_{2,n}  \|_{L^p(\Omega,\delta^\gamma)} \leq C\| \mu_{2,n} \|_{\M(\Omega,\delta^\gamma)} \leq C \| \mu^- \|_{\M(\Omega,\delta^\gamma)}. 
\end{equation}
Thus the sequences $\{u_{i,n}\}_{n \in \mathbb{N}}$, $i=1,2$, are increasing and uniformly bounded in $L^1(\Omega,\delta^\gamma)$ as well as in $L^p(\Omega,\delta^\gamma)$. By the monotonicity convergence theorem, $u_{1,n} \to u^1$ and $u_{2,n} \to u^2$ in $L^1(\Omega,\delta^\gamma)$ and  $L^p(\Omega,\delta^\gamma)$ for some $u^1, u^2 \in L^p(\Omega,\delta^{\gamma})$.  {\color{xanhle} Take $\xi \in \delta^\gamma L^\infty(\Omega)$, then $\xi \in L^r(\Omega,\delta^\gamma)$ for some $r>\frac{N+\gamma}{2s}$ and consequently $\Gbb^{\Omega}[\xi] \in C(\overline{\Omega})$ due to Proposition \ref{compact-large}.  Also,  we have $\Gbb^{\Omega}[\xi] \in \delta^{\gamma}L^{\infty}(\Omega)$,  see for instance \cite[Proposition 3.5]{ChaGomVaz_2019_2020}. It follows that $\delta^{-\gamma}\Gbb^{\Omega}[\xi] \in C(\overline{\Omega})$.} Since $\{\nu_{1,n}\}_{n \in \mathbb{N}}$ converges to $\nu^+$ in the weak$^*$ topology of $\M(\Omega)$, we have
\begin{align*}
\lim_{n \to \infty} \int_{\Omega} \Gbb^{\Omega}[\xi] d\mu_{1,n} = \lim_{n \to \infty} \int_{\Omega} \delta^{-\gamma} \Gbb^{\Omega}[\xi] d\nu_{1,n} = \int_{\Omega} \delta^{-\gamma} \Gbb^{\Omega}[\xi] d\nu_{+} = 	\int_{\Omega} \Gbb^{\Omega}[\xi] d\mu_{+}.
\end{align*}

Therefore, letting $n \to \infty$ in \eqref{eq:u_in} leads to
\begin{equation} \label{eq:u_1n} \int_{\Omega} u^1\xi dx + \int_{\Omega} (u^1)^{p}\Gbb^{\Omega}[\xi] dx = \int_{\Omega} \Gbb^{\Omega}[\xi] d\mu_{+},\quad \forall \xi \in \delta^{\gamma}L^{\infty}(\Omega).
\end{equation} 
Thus $u^1$ is the (unique) weak-dual solution of \eqref{eq:mu1n} with $\mu_{1,n}$ replaced by $\mu^+$.  Similarly,  $u^2$ is the (unique) weak-dual solution of \eqref{eq:mu1n} with $\mu_{2,n}$ replaced by $\mu^-$.

Next we will show that $\{u_n\}_{n \in \N}$ is a convergent sequence.  Indeed,  it follows from \eqref{order} that $ |u_n| \leq \max\{u_{1,n},u_{2,n}\}$.
This and \eqref{est:apriori-u1}, \eqref{est:apriori-u2} imply in that $\{u_n\}_{n \in \mathbb{N}}$ is uniformly bounded in $L^p(\Omega,\delta^\gamma)$. In addition, one can write
$$
u_n = \Gbb^{\Omega}[\mu_{1,n} - \mu_{2,n} - |u_n|^{p-1}u_n]
$$
Since $\{\mu_{i,n}\}_{n \in \mathbb{N}}$, $i=1,2$, are uniformly bounded in $\M(\Omega,\delta^\gamma)$ and $|u_n|^p$ is uniformly bounded in $L^1(\Omega,\delta^\gamma)$, by using the compactness property of $\Gbb^{\Omega}$ form $\M(\Omega,\delta^\gamma)$ into $L^1(\Omega,\delta^\gamma)$ (see Theorem \ref{theo:compactness}), we derive that, up to a subsequence, $u_n \to u$ in $L^1(\Omega,\delta^\gamma)$ and a.e. in $\Omega$. By the dominated convergence theorem, $u_n \to u$ in $L^1(\Omega,\delta^\gamma)$ and in $L^p(\Omega,\delta^\gamma)$. Passing $n \to \infty$ in \eqref{eq:u_n} yields
\begin{equation} \label{weakform-up}
\int_{\Omega} u \xi dx + \int_{\Omega} |u|^{p-1}u\Gbb^{\Omega}[\xi] dx = \int_{\Omega} \Gbb^{\Omega}[\xi] d\mu,\quad \forall \xi \in \delta^{\gamma}L^{\infty}(\Omega).
\end{equation}
Thus $u$ is a weak-dual solution to problem \eqref{eq:up}. The uniqueness follows from Kato's inequality (see Theorem \ref{thm:kato}). Estimates \eqref{Gmu+-} follow from \eqref{order}. The proof is complete.
\end{proof} \medskip

\noindent \textbf{Acknowledgments.} The work of P.-T. Huynh was supported by the Austrian Science Fund FWF under the grant DOC 78.  He gratefully thanks Prof.  Jan Slov\'{a}k for the kind hospitality during his stay at Masaryk University. P.-T. Nguyen was supported by Czech Science Foundation, Project GA22-17403S. Part of the work was carried out during the visit of P.-T. Nguyen at the University of Klagenfurt. P-T. Nguyen gratefully acknowledges the University of Klagenfurt for the hospitality.The authors would like to thank the anonymous referee for the comments which helped to improve the paper.

\appendix

\section{Proof of some results}\label{appendix:proof_of_lemma}

In this appendix,  we give the proofs of  technical results used to prove the main results.

\begin{proof}[\sc Proof of Lemma \ref{lem:technical}]\text{}

(i) Assume that \eqref{eq_G2}--\eqref{eq_G3bis} hold.  We first show that $G^{\varepsilon}_{\beta}$ is continuous in $\Omega \times \Omega$. Indeed,  by \eqref{eq_G3bis} $G^{\Omega}$ is continuous in $(\Omega \times \Omega)\backslash D_{\Omega}$.  In addition, the map $(x,y) \mapsto K^{\varepsilon}_{\beta}(|x-y|)$ is continuous since $K^{\varepsilon}_{\beta}$ is continuous on $\mathbb{R}^+$. Thus,  $G^{\varepsilon}_{\beta}$ is continuous in $(\Omega \times \Omega) \setminus D_{\Omega}$.  We prove that $G^{\varepsilon}_{\beta}$ is continuous on the diagonal $D_{\Omega}$.  For $x,y \in \Omega$ such that $0 < |x-y| < \varepsilon$,  by \eqref{eq:Green0} one has
\begin{equation}\label{eq:limitGr}
\begin{aligned}
G^{\varepsilon}_{\beta}(x,y) = G^{\Omega}(x,y) K^{\varepsilon}_{\beta}(|x-y|) \lesssim \dfrac{1}{|x-y|^{N-2s}}\cdot \dfrac{|x-y|^{\beta}}{\varepsilon^{\beta}} \leq \dfrac{|x-y|^{\beta - (N-2s)}}{\varepsilon^{\beta}}.
\end{aligned}
\end{equation}
For $(x_0,x_0) \in D_\Omega$, combining \eqref{eq:limitGr} and the fact that $\beta > N-2s$, we have
$$\lim_{(x,y) \to (x_0,x_0)}G^{\varepsilon}_{\beta}(x,y) = 0 = G^{\varepsilon}_{\beta}(x_0,x_0).$$
Hence, $G^{\varepsilon}_{\beta}$ is continuous at $(x_0,x_0) \in D_\Omega$.  This implies $G^{\varepsilon}_{\beta}$ is continuous  on $\Omega \times \Omega$. 

We prove the uniform continuity of the map $(x,y) \mapsto G^{\varepsilon}_{\beta}(x,y)/\delta(y)^{\gamma'}$ for $\gamma' < \gamma$. Let $\Psi_{\gamma'}$ be the extension of $G_{\beta}^{\varepsilon}/{\delta^{\gamma'}}$ on $\overline{\Omega} \times \overline{\Omega}$ defined by
$$
\Psi_{\gamma'}(x,y): = \left\{  \begin{aligned}
&\dfrac{G^{\varepsilon}_{\beta}(x,y)}{\delta(y)^{\gamma'}} \quad &&\text{if } (x,y) \in \Omega \times \Omega, \\
&0 \quad &&\text{elsewhere}.
\end{aligned} \right.
$$

By the continuity of $G^{\varepsilon}_{\beta}$ in $\Omega \times \Omega$,  $\Psi_{\gamma'}$ is continuous in $\Omega \times \Omega$. We prove that $\Psi_{\gamma'}$ is continuous on the boundary. By \eqref{eq:Green0} and \eqref{Keb}, for $(x,y) \in \Omega \times \Omega$, we derive
\begin{equation}\label{eq:Green_bound1}
\begin{aligned}
\dfrac{G^{\varepsilon}_{\beta}(x,y)}{\delta(y)^{\gamma'}} 
= \dfrac{G^{\Omega}(x,y)}{\delta(y)^{\gamma'}}K^{\varepsilon}_{\beta}(|x-y|) \lesssim \dfrac{\delta(x)^\gamma \delta(y)^{\gamma - \gamma'}}{|x-y|^{N+2\gamma-2s}}\cdot \left(1 \wedge \dfrac{|x-y|^{\beta}}{\varepsilon^{\beta}}\right)
\lesssim \dfrac{\delta(x)^\gamma\delta(y)^{\gamma - \gamma'}}{\varepsilon^{N+ 2\gamma - 2s}}.
\end{aligned}
\end{equation}
Here we have used the fact that $\beta > N + 2\gamma - 2s$ and the elementary inequality 
\begin{equation} \label{1tab} 
(1 \wedge t^a) \le t^b,  \quad t \ge 0, b \leq a.
\end{equation}
Therefore, since $\gamma>\gamma'$,  we conclude that for any $(x_0,y_0) \in \partial \Omega \times \overline{\Omega}$ or $(x_0,y_0) \in \overline{\Omega} \times \partial \Omega$,
$$ \lim_{\substack{(x,y) \to (x_0,y_0)\\ (x,y) \in \Omega \times \Omega} } \Psi_{\gamma'}(x,y) =  \lim_{\substack{(x,y) \to (x_0,y_0)\\ (x,y) \in \Omega \times \Omega}} \dfrac{G^{\varepsilon}_{\beta}(x,y)}{\delta(y)^{\gamma'}} = 0 = \Psi_{\gamma'}(x_0,y_0).
$$
Thus, $\Psi_{\gamma'}$ is continuous on $\overline \Omega \times \overline \Omega$. It follows that $G^{\varepsilon}_{\beta}/\delta^{\gamma'}$ is uniformly continuous on $\Omega \times  \Omega$.

(ii) Assume, in addition, that \eqref{eq_G4} holds.  We find that the map
\begin{equation} \label{eq:bd-a} (x,y) \mapsto \dfrac{G^{\varepsilon}_{\beta}(x,y)}{\delta(y)^{\gamma}} \quad \text{is continuous on } \Omega \times \Omega.
\end{equation} 
Also by using the same argument leading to \eqref{eq:Green_bound1}, we obtain
$$
\dfrac{G^{\varepsilon}_{\beta}(x,y)}{\delta(y)^{\gamma}} \lesssim \dfrac{\delta(x)^\gamma}{\varepsilon^{N+ 2\gamma - 2s}}, \quad \forall (x,y) \in \Omega \times \Omega.
$$
Therefore, for any $x_0 \in \partial \Omega$ and $y_0 \in \overline \Omega$,
\begin{equation} \label{eq:bdryx} \lim_{(x,y) \to (x_0,y_0)}\dfrac{G^{\varepsilon}_{\beta}(x,y)}{\delta(y)^{\gamma}} = 0.
\end{equation}
Next, by assumption \eqref{eq_G4}, for $x_0 \in \Omega$ and $z_0 \in \partial \Omega$, we have
\begin{equation}\label{eq:boundary_limit_2}
	\lim_{(x,y) \to (x_0,z_0)} \dfrac{G^{\varepsilon}_{\beta}(x,y)}{\delta(y)^{\gamma}} = \lim_{(x,y) \to (x_0,z_0)} \dfrac{G^{\Omega}(x,y)}{\delta(y)^{\gamma}} K^{\varepsilon}_{\beta}(|x-y|) = M^{\Omega}(x_0,z_0)K^{\varepsilon}_{\beta}(|x_0-z_0|).
\end{equation}
Note that the function on the right hand side of \eqref{eq:boundary_limit_2} is continuous in $\Omega \times \partial \Omega$. 
Furthermore, by \eqref{M-est}, \eqref{Keb} (with $\beta > N - 2s + 2\gamma$) and \eqref{1tab}, we obtain
\begin{align*} M^{\Omega}(x_0,z_0)K^{\varepsilon}_{\beta}(|x_0-z_0|) \lesssim \frac{\delta(x_0)^\gamma}{|x_0-z_0|^{N+2\gamma-2s}} \left( 1 \wedge \frac{|x_0-z_0|^\beta}{\varepsilon^{\beta}} \right) \lesssim \frac{\delta(x_0)^\gamma}{\varepsilon^{N+2\gamma-2s}},
\end{align*}
which implies
\begin{equation} \label{bdry-bdry} \lim_{x_0 \to \partial \Omega} M^{\Omega}(x_0,z_0)K^{\varepsilon}_{\beta}(|x_0-z_0|) = 0. 
\end{equation}
Let $\Psi_{\gamma}$ be the extension of $G^{\varepsilon}_{\beta}/\delta^{\gamma}$ on $\overline \Omega \times \overline \Omega$ defined by
$$
\Psi_{\gamma}(x,y): = \left\{  \begin{aligned}
&\dfrac{G^{\varepsilon}_{\beta}(x,y)}{\delta(y)^{\gamma}} \quad &&\text{if } (x,y) \in \Omega \times \Omega, \\
&M^{\Omega}(x,y)K_{\beta}^{\varepsilon}(|x-y|) &&\text{if } (x,y) \in \Omega \times \partial \Omega,\\
&0 &&\text{elsewhere.}
\end{aligned} \right.
$$
We infer from \eqref{eq:bd-a}--\eqref{bdry-bdry} that $\Psi_{\gamma}$ is continuous on $\overline \Omega \times \overline \Omega$. Therefore, $\Psi_{\gamma}$ is uniformly continuous on $\overline \Omega \times \overline \Omega$. Consequently, 
$G^{\varepsilon}_{\beta}/\delta^{\gamma}$ is uniformly continuous in $\Omega \times \Omega$.  The proof is complete.
\end{proof}

\section{Boundary singularities} \label{sec:boundarysolution}

This section is devoted to a short discussion on the singularities of solutions of \eqref{eq:semilinear_absorption}.  We assume that \eqref{eq_L1}--\eqref{eq_L4} and \eqref{eq_G2}--\eqref{eq_G4} hold and we will show that any isolated boundary singularity can be obtained as the limit of isolated interior singularities. Recall that the Martin kernel $M^{\Omega}: \Omega \times \partial \Omega \to \R$ is given by
$$  M^{\Omega}(x,z) =\lim_{\Omega \ni y \to z}\frac{G^\Omega(x,y)}{\delta(y)^\gamma}, \quad x \in \Omega, z \in \partial \Omega.
$$
Recall also that $p^*=\frac{N+\gamma}{N+\gamma-2s}$.

\begin{proposition}
	Assume that $1 < p < p^*$. Then for any $z \in \partial \Omega$, 
	\begin{equation} \label{GMp}
		\lim_{\Omega \ni  x \to z}\frac{\Gbb^\Omega[M^\Omega(\cdot,z)^p](x)}{M^\Omega(x,z)} = 0. 
	\end{equation} 	
\end{proposition}

\begin{proof} It can be seen from \eqref{M-est} that 
	$$
	\Gbb^{\Omega}[M^{\Omega}(\cdot,z)^p](x) = \int_{\Omega} G^{\Omega}(x,y) M^{\Omega}(y,z)^p dy \asymp \int_{\Omega} G^{\Omega}(x,y) \left( \dfrac{\delta(y)^{\gamma}}{|y-z|^{N-2s+2\gamma}} \right)^p dy.
	$$
	This, together with \eqref{M-est} again, implies
	$$
	\dfrac{\Gbb^{\Omega}[M^{\Omega}(\cdot,z)^p](x)}{M^{\Omega}(x,z)} \asymp \delta(x)^{-\gamma} |x-z|^{N-2s+2\gamma}\int_{\Omega} G^{\Omega}(x,y) \left( \dfrac{\delta(y)^{\gamma}}{|y-z|^{N-2s+2\gamma}} \right)^p dy.
	$$
	Put
	$$
	\Omega_1:=\Omega \cap B\left(x,\dfrac{|x-z|}{2}\right), \quad \Omega_2 :=\Omega \cap B\left(z,\dfrac{|x-z|}{2}\right), \quad \Omega_3:=\Omega \setminus (\Omega_1 \cup\Omega_2)
	$$
	and
	$$I_k :=  \delta(x)^{-\gamma} |x-z|^{N-2s+2\gamma}\int_{\Omega_k} G^{\Omega}(x,y) \left( \dfrac{\delta(y)^{\gamma}}{|y-z|^{N-2s+2\gamma}} \right)^pdy, \quad k = 1,2,3.$$
	
	We first estimate $I_1$.  For any $y \in \Omega_1$, we have $2|y-z| \ge |x-z|$ and obviously, $\delta(y) \le |y-z|$.  Furthermore, since $p < p^*$,  there exists $\alpha > 0$ such that
	$$(p-1)(N-2s + 2\gamma) < \alpha < \left( \gamma(p-1)+2s \right)\wedge \left( p(N-2s+2\gamma) \right).$$
	Then we find
	\begin{align*}
		I_1 &\lesssim \delta(x)^{-\gamma} |x-z|^{N-2s+2\gamma} \int_{\Omega_1} G^{\Omega}(x,y) \dfrac{\delta(y)^{\gamma p - \alpha}}{|y-z|^{p(N-2s+2\gamma)-\alpha}} dy\\
		&\lesssim \delta(x)^{-\gamma} |x-z|^{(N-2s+2\gamma)(1-p)+\alpha} \int_{\Omega_1} G^{\Omega}(x,y) \delta(y)^{\gamma p -\alpha} dy \\
		&\lesssim \delta(x)^{-\gamma} |x-z|^{(N-2s+2\gamma)(1-p)+\alpha} \Gbb^{\Omega}[\delta^{\gamma p - \alpha}](x).
	\end{align*}
	We have $\gamma p - \alpha > \gamma -2s$ since $\alpha < \gamma(p-1)+2s$.  Thus by \cite[Theorem 3.4]{Aba_2019}, $\Gbb^{\Omega}[\delta^{\gamma p - \alpha}] \asymp \delta^{\gamma}$. This implies
	$$I_1 \lesssim \delta(x)^{-\gamma} |x-z|^{(N-2s+2\gamma)(1-p)+\alpha}  \delta(x)^{\gamma} \le |x-z|^{(N-2s+2\alpha)(1-p)+\alpha}.$$
	Since $(N-2s+2\gamma)(1-p) +\alpha > 0$, it follows that
	\begin{equation} \label{limI1} \lim_{x \to z}I_1 = 0.
	\end{equation}
	
	We next estimate $I_2$.  For $y \in \Omega_2$, we have $|x-z| \leq 2|x-y|$ and $\delta(y) \le |y-z|$. This, together with the fact that
	\begin{equation}\label{eq:Gestimate2}
		G^{\Omega}(x,y) \lesssim \dfrac{\delta(x)^{\gamma}\delta(y)^{\gamma}}{|x-y|^{N-2s+2\gamma}}, \quad (x,y) \in (\Omega \times \Omega) \backslash D_{\Omega},
	\end{equation} implies
	\begin{align*}
		I_2 &\lesssim \delta(x)^{-\gamma} |x-z|^{N-2s+2\gamma} \int_{\Omega_2} \dfrac{\delta(x)^{\gamma}\delta(y)^{\gamma}}{|x-y|^{N-2s+2\gamma}}\cdot \left( \dfrac{\delta(y)^{\gamma}}{|y-z|^{N-2s+2\gamma}}\right)^p dy \\
		&\lesssim \int_{\Omega_2} \dfrac{1}{|y-z|^{(N-2s+\gamma)p - \gamma}} dy \lesssim \int_0^{|x-z|} t^{N-1+\gamma - (N+\gamma-2s)p}dt \\
		&= \frac{1}{N+\gamma - (N+\gamma-2s)p}|x-z|^{N+\gamma - (N+\gamma-2s)p}.
	\end{align*}
	Since $p<p^*$, it follows that
	\begin{equation} \label{limI2}  \lim_{x \to z}I_2 = 0.
	\end{equation}
	
	Finally, we estimate $I_3$. For every $y \in \Omega_3$,  we have $|y-z| \leq 3|x-y|$ and $\delta(y) \le |y-z|$.  Combining with \eqref{eq:Gestimate2} again, we obtain
	\begin{equation}\label{eq:I3_1}
		\begin{aligned}
			I_3 
			&\lesssim \delta(x)^{-\gamma} |x-z|^{N-2s+2\gamma} \int_{\Omega_3} \dfrac{\delta(x)^{\gamma}\delta(y)^{\gamma}}{|x-y|^{N-2s+2\gamma}}\cdot \left( \dfrac{\delta(y)^{\gamma}}{|y-z|^{N-2s+2\gamma}}\right)^p dy\\
			&\lesssim |x-z|^{N-2s+2\gamma} \int_{\Omega_3}|y-z|^{-(N-2s+\gamma)(p+1)}dy \\
			&\lesssim |x-z|^{N-2s+2\gamma} \int_{\frac{1}{2}|x-z|}^{3\mathrm{diam}(\Omega)}t^{N-1-(N-2s+\gamma)(p+1)}dt.
		\end{aligned}
	\end{equation}
	Notice that for $a > 0$ and $b \in \mathbb{R}$
	$$\lim_{ w \to 0^+} w^a \int_w^1 t^{b} dt =  0 \text{ if } a+ b +1  > 0.$$
	We note that $(N-2s+2\gamma) +  \left[ N-1-(N-2s+\gamma)(p+1)\right] + 1 > 0$
	since $p < p^*$.  Therefore from \eqref{eq:I3_1}, we conclude that
	\begin{equation} \label{limI3} \lim_{x \to z}I_3 = 0.
	\end{equation}
	Thus combining \eqref{limI1},  \eqref{limI2} and \eqref{limI3}, we obtain \eqref{GMp}.
\end{proof}

\begin{proposition} \label{bdry-iso} Assume that \eqref{eq_L1}--\eqref{eq_L4} and \eqref{eq_G2}--\eqref{eq_G4} hold, $1<p<p^*$ and $z \in \partial \Omega$. Then there exists a unique function $u \in L^1(\Omega,\delta^\gamma)$ such that
	\begin{equation} \label{eq:boundary} u + \Gbb^{\Omega}[u^p] = M^\Omega(\cdot,z) \quad \text{a.e. in } \Omega.
	\end{equation}
	Moreover,  there holds
	\begin{equation} \label{asymp} \lim_{\Omega \ni x \to z}\frac{u(x)}{M^{\Omega}(x,z)}=1.
	\end{equation}
\end{proposition}
Roughly speaking, the function $u$ in Theorem \ref{bdry-iso} could be understood as a `solution' to the boundary problem
\begin{equation} \label{eq:bdry}
	\left\{ \begin{aligned}
		\L u + u^p &= 0&&\text{ in }\Omega, \\
		u  &= \delta_z &&\text{ on }\partial \Omega, \text{ or in }\R^N \backslash \overline{\Omega} \text{ (if applicable)}.
	\end{aligned} \right.
\end{equation}

Although the problem of boundary singularities of solutions to \eqref{eq:bdry} is interesting, it is not a topic that we want to consider and dwell on in this paper. 

\begin{proof}[{\sc Proof of Theorem \ref{bdry-iso}}.]
	Let $\{z_n\}_{n \in \mathbb{N}} \subset \Omega$ be a sequence converging to $z \in \partial \Omega$. Put $\mu_n = \delta^{-\gamma} \delta_{z_n}$ then $\mu_n \in \M(\Omega,\delta^\gamma)$ and $\| \mu_n \|_{\M(\Omega,\delta^\gamma)}=1$. Since $p \in (1,p^*)$, for every $n \in \mathbb{N}$, by Corollary \ref{cor:measuredata-sub},  there exists a unique positive function $u_n \in L^1(\Omega,\delta^{\gamma})$ satisfying
	\begin{equation} \label{un-repr} u_n + \Gbb^\Omega[u_n^p] = \Gbb^\Omega[\mu_n]. 
	\end{equation}
	
	Since $\| \mu_n \|_{\M(\Omega,\delta^\gamma)}=1$ and $\Gbb^\Omega: \M(\Omega,\delta^\gamma) \to L^1(\Omega,\delta^\gamma)$ is compact (see Theorem \ref{theo:compactness}), we know that up to a subsequence $\{ \Gbb^\Omega[\mu_n]\}_{n \in \mathbb{N}}$ converges in $L^1(\Omega,\delta^\gamma)$ and  a.e. in $\Omega$ to a function $v_1 \in L^1(\Omega,\delta^{\gamma})$.  This, combined with assumption \eqref{eq_G4}, yields
	$$ v_1(x) = \lim_{n \to \infty}\Gbb^\Omega[\mu_n](x) =   \lim_{n \to \infty}\frac{G^\Omega(x,z_n)}{\delta(z_n)^\gamma} = M^\Omega(x,z) \quad \text{for a.e. } x \in \Omega.
	$$
	
	Next, we infer from \eqref{un-repr} that 
	$ 0 \le u_n \le \Gbb^{\Omega}[\mu_n]$ for all $n \in \mathbb{N}$. 
	Hence for all $q \in (1, p^*)$, there holds
	\begin{equation} \label{unLq} \| u_n \|_{L^q(\Omega,\delta^\gamma)} \leq  \| \Gbb^{\Omega}[\mu_n] \|_{L^q(\Omega,\delta^\gamma)} \le C(N,\Omega,s,\gamma,q),\quad \forall n \in \N.
	\end{equation}
	Taking \eqref{unLq} with $q=p$ and taking into account the fact that $\Gbb^\Omega: L^1(\Omega,\delta^\gamma) \to L^1(\Omega,\delta^\gamma)$ is compact, we see that there exists a function $v_2 \in L^1(\Omega,\delta^\gamma)$  such that, up to a subsequence, $\{ \Gbb^\Omega[u_n^p]\}_{n \in \mathbb{N}}$ converges to $v_2$  in $L^1(\Omega,\delta^\gamma)$ and a.e. in $\Omega$. Put $u=M^\Omega(\cdot,z) - v_2$ then from \eqref{un-repr}, we realize that $u_n \to u$ in $L^1(\Omega,\delta^\gamma)$ and a.e. in $\Omega$. 
	
	In light of \eqref{unLq} and H\"older inequality, $\{u_n^p\}_{n \in \mathbb{N}}$ is equi-integrable in $L^1(\Omega,\delta^\gamma)$. Moreover, $u_n^p \to u^p$ a.e. in $\Omega$. Therefore, by Vitali's convergence theorem, $u_n^p \to u^p$ in $L^1(\Omega,\delta^\gamma)$. Consequently, by Corollary \ref{cor:stability}, $\Gbb^{\Omega}[u_n^p] \to \Gbb^{\Omega}[u^p]$ in $L^1(\Omega,\delta^\gamma)$ and a.e. in $\Omega$. Gathering the above facts and letting $n \to \infty$ in \eqref{un-repr} yields \eqref{eq:boundary}. In addition, the uniqueness follows from Kato's inequality \eqref{eq:kato_abs}.
	
	Next we prove \eqref{asymp}. From \eqref{eq:boundary}, we obtain $M^\Omega(x,z) - \Gbb^{\Omega}[M^\Omega(\cdot,z)^p](x) \leq u(x) \leq M^\Omega(x,z)$, 
	which implies
	$$ 1 - \frac{\Gbb^{\Omega}[M^\Omega(\cdot,z)^p](x)}{M^\Omega(x,z)} \leq \frac{u(x)}{M^\Omega(x,z)} \leq 1.$$
	This and \eqref{GMp} imply \eqref{asymp}.	The proof is complete.
\end{proof}

\section{Compact embeddings}\label{appendix:compact}

It was shown in \cite{BonFigVaz_2018} that, under assumption \eqref{eq_G2}, the map $\Gbb^{\Omega}: L^2(\Omega) \to L^2(\Omega)$ is compact. We note that it could be proved alternatively, under assumptions \eqref{eq_L1}--\eqref{eq_L4}, by using the compact embedding $\mathbb{H}(\Omega)\hookrightarrow \hookrightarrow L^2(\Omega)$. The question of compact embedding has also been raised in \cite{ChaGomVaz_2019_2020} in the case of the space $\Hbb^2(\Omega)$. We give below an affirmative answer to this question for a general spaces defined as 
\begin{align*}
\Hbb^{\alpha} (\Omega) := \left\{ u = \sum_{n=1}^{\infty} \widehat{u_n} \varphi_n: \sum_{n=1}^{\infty}(\lambda_n)^{\alpha} \left| \widehat{u_n}\right|^2 < \infty \right\}, \alpha > 0,
\end{align*}
where $0<\lambda_1 < \lambda_2 \leq \ldots \lambda_n \to +\infty$ is the sequence of eigenvalues of $\L$ in $\Omega$ with zero Dirichlet boundary or exterior condition.

	\begin{proposition}\label{prop:compact} Suppose that \eqref{eq_L1}--\eqref{eq_L2} and \eqref{eq_G2} hold. Then $\Hbb^{\alpha}(\Omega)$ is compactly embedded into $L^2(\Omega)$ for every $\alpha > 0$.
	\end{proposition}
	
	\begin{proof}Consider a bounded set $V$ in $\Hbb^{\alpha}(\Omega)$. For any $u \in V$ and $k \in \mathbb{N}$, one has
		\begin{align*}
			C \ge \sum_{n=1}^{\infty}(\lambda_n)^{\alpha} \left| \widehat{u_n}\right|^2 \ge \sum_{n= k}^{\infty}(\lambda_n)^{\alpha} \left| \widehat{u_n}\right|^2  \ge \sum_{n= k}^{\infty}(\lambda_k)^{\alpha} \left| \widehat{u_n}\right|^2,
		\end{align*}
		which implies 
		$\sum_{n= k}^{\infty}\left| \widehat{u_n}\right|^2 < C(\lambda_k)^{-\alpha}$.
		For every $\varepsilon > 0$, one can choose $k \in \mathbb{N}$ large enough such that $C(\lambda_k)^{-\alpha} < \varepsilon$. Such a $k$ exists since $\lambda_n \to \infty$ as $n \to \infty$. This implies
		$\sum_{n= k}^{\infty}\left| \widehat{u_n}\right|^2 < \varepsilon,\quad \forall u \in V$.
		One can see that \cite[Proposition 1.2.39]{DraMil_2013} yields $V$ is compact in $L^2(\Omega)$. We complete the proof.
	\end{proof}

\end{document}